\documentclass[10pt]{article}
\usepackage{graphicx} %
\usepackage{amsmath}
\usepackage{amsthm}
\usepackage{amssymb}   
\usepackage{mathrsfs}  
\usepackage{enumitem}  
\usepackage[
    margin=1in
]{geometry}
\usepackage[colorlinks,
            linkcolor=red,
            anchorcolor=blue,
            citecolor=green
            ]{hyperref}
\usepackage{chngcntr}  
\numberwithin{equation}{section}
            
\usepackage{mathtools}

\usepackage{algorithm}
\usepackage{algorithmic}
\usepackage{natbib}

\usepackage{microtype}
\usepackage{graphicx}
\usepackage{subfigure}
\usepackage{booktabs} %
\usepackage{makecell}

\usepackage{hyperref}

\newcommand{\bx}{\mathbf{x}}
\newcommand{\by}{\mathbf{y}}
\newcommand{\bz}{\mathbf{z}}
\newcommand{\ba}{\mathbf{a}}
\newcommand{\bb}{\mathbf{b}}
\newcommand{\bc}{\mathbf{c}}

\newcommand{\bg}{\mathbf{g}}
\newcommand{\msfx}{\mathsf{x}}

\newcommand{\xoutk}{\mathbf{x}_k^{\mathsf{out}}}
\newcommand{\xoutK}{\mathbf{x}_K^{\mathsf{out}}}
\newcommand{\youtK}{\mathbf{y}_K^{\mathsf{out}}}

\newcommand{\R}{\mathbb{R}}

\renewcommand{\phi}{\varphi}

\newcommand{\Gap}{\mathrm{Gap}}

\newcommand{\Eb}{\mathbb{E}}
\newcommand{\Lcal}{\mathcal{L}}
\newcommand{\Ftilde}{\widetilde{F}}

\newcommand{\good}{\mathsf{Good}}
\newcommand{\bad}{\mathsf{Bad}}
\newcommand{\resid}{\mathcal{R}}
\newcommand{\res}{\mathsf{Res}}

\DeclareFontFamily{OT1}{pzc}{}
\DeclareFontShape{OT1}{pzc}{m}{it}{<-> s * [1.200] pzcmi7t}{}
\DeclareMathAlphabet{\mathpzc}{OT1}{pzc}{m}{it}

\DeclareMathOperator{\dist}{dist}

\DeclareMathOperator{\prox}{prox}

\usepackage{xcolor,colortbl}

\hypersetup{
  linkcolor  = blue,
  citecolor  = teal,
  colorlinks = true,
}
\makeatletter
\newcommand\blfootnote[1]{%
  \begingroup
  \renewcommand{\@makefntext}[1]{\noindent\makebox[1.8em][r]#1}
  \renewcommand\thefootnote{}\footnote{#1}%
  \addtocounter{footnote}{-1}%
  \endgroup
}
\makeatother

\usepackage[capitalize,noabbrev]{cleveref}

\theoremstyle{plain}
\newtheorem{theorem}{Theorem}[section]
\newtheorem{proposition}[theorem]{Proposition}
\newtheorem{lemma}[theorem]{Lemma}
\newtheorem{corollary}[theorem]{Corollary}
\theoremstyle{definition}

\newtheorem{assumption}[theorem]{Assumption}
\theoremstyle{remark}
\newtheorem{remark}[theorem]{Remark}
\theoremstyle{fact}
\newtheorem{fact}[theorem]{Fact}
\newtheorem{example}[theorem]{Example}

\crefname{assumption}{assumption}{assumptions}
\Crefname{assumption}{Assumption}{Assumptions}

\usepackage[textsize=tiny]{todonotes}

\title{Towards Weaker Variance Assumptions for Stochastic Optimization}
\author{Ahmet Alacaoglu\footnote{Department of Mathematics, University of British Columbia. \url{alacaoglu@math.ubc.ca}} \and Yura Malitsky\footnote{Faculty of Mathematics, University of Vienna. \url{yurii.malitskyi@univie.ac.at}} \and Stephen J. Wright\footnote{Department of Computer Sciences, University of Wisconsin--Madison. \url{swright@cs.wisc.edu}}}
\date{}

\begin{document}

\maketitle

\begin{abstract}
We revisit a classical assumption for analyzing stochastic gradient algorithms where the squared norm of the stochastic subgradient (or the variance for smooth problems) is allowed to grow as fast as the squared norm of the variable. We contextualize this assumption in view of its inception in the 1960s, its seemingly independent appearance in the recent literature, its relationship to weakest-known variance assumptions for analyzing stochastic gradient algorithms, and its relevance in deterministic problems for non-Lipschitz nonsmooth convex optimization. We extend a connection recently made between this assumption and the Halpern iteration. For convex nonsmooth, and potentially stochastic, optimization, we analyze horizon-free, anytime algorithms with last-iterate rates. For problems beyond simple constrained optimization, such as convex problems with functional constraints or regularized convex-concave min-max problems, we obtain rates for optimality measures that do not require boundedness of the feasible set.  
\end{abstract}

\section{Introduction}
We consider the prototypical problem
\begin{equation} \label{eq:proto}
    \min_{\bx\in X} f(\bx),
\end{equation}
where $X\subseteq \mathbb{R}^d$ is convex and closed and $f\colon\mathbb{R}^d\to\mathbb{R}$ is convex but not necessarily smooth. 
For stochastic gradient descent (SGD) and related methods, it is commonly assumed that we have access to an unbiased \emph{oracle} $\widetilde{\nabla} f$, that is, 
\begin{equation}\label{eq: unbiased}
    \mathbb{E}[\widetilde{\nabla} f(\bx)] \in \partial f(\bx).
\end{equation}
Typically, at iteration $k$ of these stochastic methods, the quantity $\widetilde{\nabla} f(\bx_k)$ is used to construct a step from $\bx_k$ to $\bx_{k+1}$. 
Many works have proved convergence results for the resulting iterative process, starting with \citet{robbins1951stochastic} and continuing to the present day.
Most of these works require $\widetilde{\nabla} f$ to satisfy certain additional properties having to do with its variance.
Several such assumptions, despite being ubiquitous, are unsatisfactory as they exclude fundamental problems in data science --- including but not limited to least squares, basis pursuit, and quadratic programming --- and certain obvious choices of oracle. 
In this paper, we sketch the history of these assumptions and explore the relationships between them.
Focusing on the weakest (that is, least restrictive) of these assumptions to our knowledge, we build on a recently proposed conjunction with the Halpern iteration to derive stronger convergence results, including results concerning the last iterate in the sequence as well as optimality measures that do not require boundedness of the feasible set with algorithmic parameters independent of a fixed horizon.
We consider contexts that range from constrained minimization to functionally constrained minimization and min-max problems.

{In the remainder of this introductory section, we focus mainly on the unconstrained case $X = \R^d$. 
The constrained problem in \eqref{eq:proto} with $X\neq\mathbb{R}^d$ is the focus of \Cref{sec: main_min}, where we describe convergence results for the objective suboptimality measure $f(\bx^\mathsf{out}) - f(\bx^\star)$ (where $\bx^\star$ is the minimizer of $f$), where $\bx^\mathsf{out}$ is either the weighted-average of the iterates or the last iterate, under variance assumptions on the stochastic gradient oracle that are weaker than those conventionally used. 
In \Cref{sec: main_minmax}, we consider the constrained optimization problem (in which $X$ of \eqref{eq:proto} is defined by a finite set of algebraic inequalities) and the min-max (saddle-point) problem. 
The Blum-Gladyshev condition we discuss below  (\eqref{eq: bg} and \eqref{eq: bg0}) plays an important role in our analyses of these cases too, allowing us to derive convergence rates for certain measures of optimality that do not require boundedness of the feasible set. We focus on the latter problems because unbounded feasible sets are common for these cases, see \Cref{sec: main_minmax} for details.}

\textbf{A Tale of Two Assumptions.} 
A standard assumption on the oracle $\widetilde{\nabla} f$ is the so-called \emph{bounded stochastic subgradient} assumption, which states that there exists $G<\infty$ such that
\begin{equation}\label{eq: ffm4}
    \mathbb{E}\| \widetilde{\nabla} f(\bx) \|^2 \leq G^2.
\end{equation}
When $f$ is smooth, \eqref{eq: ffm4} is commonly relaxed to the \emph{bounded-variance} assumption, requiring 
\begin{equation} \label{eq:bva}
    \mathbb{E}\| \widetilde{\nabla} f(\bx) -\nabla f(\bx)\|^2 \leq G^2,
\end{equation}
for  $G<\infty$.
For simplicity, we focus on the former, but our discussions apply equally to the latter.

A slightly weaker variant of \eqref{eq: ffm4} was used in the foundational reference for SGD --- \citet[Eq. (4)]{robbins1951stochastic} --- for the purpose of asymptotic analysis.
(Our variant \eqref{eq: ffm4} leads to a simpler discussion.)  
Not long after the publication of \cite{robbins1951stochastic}, a weaker assumption appeared in the works of \citet[Eq. A]{blum1954approximation} and \citet[Theorem 1, condition 2)]{gladyshev1965stochastic}, namely,
\begin{equation}\label{eq: bg}\tag{BG}
    \mathbb{E}\| \widetilde{\nabla} f(\bx) \|^2 \leq \bar{B}^2\|\bx\|^2 + \bar{G}^2,
\end{equation}
with finite $\bar{B}$ and $\bar{G}$. 
Again, the purpose of the assumption was to facilitate the analysis of the asymptotic behavior of SGD. 
(Note that \cite{blum1954approximation} also requires the bounded-variance assumption \eqref{eq:bva}.)
To emphasize its origins, we refer to \eqref{eq: bg} as the Blum-Gladyshev (BG) assumption.
Even though the works we cited constitute the earliest appearance of \eqref{eq: bg} to our knowledge, a similar assumption in the context of subgradient methods also appeared in the  work of \citet{cohen1983decomposition}.

More recently, the assumption \eqref{eq: bg} has been used or mentioned in a number of works, including \citep[Assumption 1]{wang2016stochastic}, \citep[Assumption 4]{cui2021analysis}, \citep{domke2023provable, asi2019stochastic,jacobsen2023unconstrained,telgarsky2022stochastic}.
Interestingly, \cite{wang2016stochastic} and \cite{cui2021analysis} use the weaker assumption \eqref{eq: bg} for purposes of asymptotic analysis, but then make an additional assumption on compactness of the feasible set $X$ to obtain convergence rates; see \cite[Theorem 2]{cui2021analysis}, \cite[Theorem 2]{wang2016stochastic}. 
Our starting point and motivation for this paper was the appearance of this assumption in a recent work  of \citet{neu2024dealing}, discussed further below.

\textbf{A third assumption.} 
As  discussed above, the classical literature on SGD focused on asymptotic convergence guarantees \citep{robbins1951stochastic,blum1954approximation,gladyshev1965stochastic,robbins1971convergence}.
The past two decades have seen a surge of interest in {\em nonasymptotic} analyses of SGD and related methods that make use of  the assumption \eqref{eq: ffm4} for non-strongly convex optimization; \cite{nemirovski1983problem,nemirovski2009robust} are representative examples. 
There has long been a recognition that \eqref{eq: ffm4} is overly restrictive: It does not even hold for linear least squares regression (for which, by contrast,  \eqref{eq: bg} is natural). 
The review article by \citet{bottou2018optimization} popularized the following generalization of \eqref{eq: ffm4}, for smooth $f$:
\begin{equation*}
    \mathbb{E} \| \widetilde{\nabla} f(\bx)\|^2 \leq c^2 + b^2 \| \nabla f(\bx)\|^2
\end{equation*}
(see also \cite[Eq. (1.6)]{bertsekas2000gradient}).
A further relaxation is
\begin{equation}\label{eq: abc}\tag{ABC}
    \mathbb{E} \| \widetilde{\nabla} f(\bx)\|^2 \leq c^2 + b^2 \| \nabla f(\bx)\|^2 + a^2 (f(\bx) - f(\bx^\star)),
\end{equation}
which is the so-called \emph{ABC condition} considered in \cite[Assumption 2]{khaled2022better} for nonconvex problems and originally appeared in \cite{polyak1973pseudogradient}. 
A variant of this assumption with  $b\equiv 0$ is utilized for convex problems in \cite{gorbunov2020unified,
  khaled2023unified,ilandarideva2023accelerated} and for nonsmooth problems in \cite{grimmer2019convergence}. 
  
We show that the classical assumption \eqref{eq: bg} is more general in the sense that it is implied by \eqref{eq: abc}. This claim follows from smoothness of $f$ in the \emph{unconstrained} case, since we have 
\begin{align*}
    f(\bx) - f(\bx^\star) \leq \frac{L}{2} \| \bx-\bx^\star\|^2 \text{\qquad and\qquad}
    \| \nabla f(\bx)\|^2 \leq L^2 \| \bx-\bx^\star\|^2,
\end{align*}
where the first inequality is the \emph{descent lemma} (see, e.g., \cite[Lemma 1.2.3]{nesterov2018lectures})
applied at $\bx$ and $\bx^\star$ and the second is Lipschitzness
  of $\nabla f$ since $\nabla f(\bx^\star) = 0$.
  (The same implication holds in the constrained case \eqref{eq:proto} or the additively composite case, which we do not discuss further, for the sake of simplicity.) 

Although 
the main advantage of \eqref{eq: bg} is that it can be applied readily to constrained problems and min-max optimization,
we describe a natural example in the unconstrained case $X = \R^d$ where \eqref{eq: bg} holds but \eqref{eq: abc} does not. 
\begin{example}
    Consider the objective $f(\bx) = \frac{1}{2d}\langle \bx, Q\bx \rangle +
    \langle \bc, \bx \rangle$, with a symmetric positive semidefinite
    matrix $Q\in \R^{d\times d}$ and $\bc\in\R^{d}$, and the stochastic
    gradient oracle
    \[\widetilde{\nabla} f(\bx) = Q_{:i} x_i + \bc,\]
    where $i\sim \mathrm{Unif}\{1, \dots, d\}$
    and $Q_{:i}$ denotes the $i$-th column of $Q$. Then for any $\bar \bx$ such that $Q\bar\bx=0$, we have, on one hand, that $f(\bar\bx) - f(\bx^\star) = \langle \bc, \bar\bx-\bx^\star\rangle - \frac{1}{2d}\langle \bx^\star, Q\bx^\star\rangle$ and $\nabla f(\bar\bx) = \bc$. On the other hand, 
    \begin{equation*}
        \mathbb{E}\|\widetilde{\nabla} f(\bar\bx) \|^2 = \frac{1}{d} \sum_{i=1}^d \| Q_{:i} \bar x_i+\bc\|^2.
      \end{equation*}
    Hence, for any $\bar\bx$ such that $Q\bar\bx=0$, the left-hand side of
    \eqref{eq: abc} grows quadratically in $\bar\bx$ and the right-hand side
    grows linearly in $\bar \bx$. As a result, there cannot exist constants $a,
    b, c$ for which \eqref{eq: abc} holds for all $\bx$. The condition
    \eqref{eq: bg} holds trivially in this case.$\hfill\blacklozenge$
\end{example}

Due to the relationship described above between \eqref{eq: bg} and \eqref{eq: abc}, a special case of the result of \cite{neu2024dealing} (see also \cite{domke2023provable} for an earlier reference) has, to our knowledge, the weakest assumptions on the variance for stochastic convex optimization problems while obtaining the optimal convergence rate. 

Given $f(\bx)=\mathbb{E}_{\xi\in\Xi}[\widetilde{f}(\bx, \xi)]$, a sufficient condition often used for \eqref{eq: abc} with $b\equiv 0$ is convexity and smoothness of the functions $\bx\mapsto \widetilde{f}(\bx, \xi)$ {\em for every $\xi$} \citep{garrigos2023handbook}.
This requirement is reasonable in the finite-sum case:
    $f(\bx) = \frac{1}{n} \sum_{i=1}^n f_i(\bx)$ and
each $f_i$ is assumed to be convex and smooth, see for example \cite[Lemma 4.20]{garrigos2023handbook}. 
(This assumption is also used in \cite{moulines2011non}.)
In contrast, \eqref{eq: bg} does not require such conditions: It can be true even when individual functions $f_i$ are nonconvex or nonsmooth, provided that the sum is convex and satisfies \eqref{eq: bg}.
This setting is often referred to as \emph{sum-of-nonconvex problems}; see for example \cite{allen2016improved}.

There is another line of work, not immediately related to ours,   that focuses on
stochastic optimization with \emph{heavy-tailed noise} as an alternative relaxation of the bounded-variance assumption \citep{gorbunov2023high,nguyen2023improved,gurbuzbalaban2021heavy}.

\subsection{Analyses under the Blum-Gladyshev Assumption}
When $f$ is strongly convex, the analysis of SGD-type methods under the assumption \eqref{eq: bg} simplifies significantly since the additional error term coming from \eqref{eq: bg} can be canceled by making use of strong convexity; see a textbook result in \cite[Section 5.4.3]{wright2022optimization}.
Several works focused on similar settings with strong convexity-type assumptions,
see for example \cite{needell2014stochastic,moulines2011non,gower2019sgd,gorbunov2022stochastic,vlatakis2024stochastic,dieuleveut2020bridging,karandikar2023convergence}. 
We focus on \emph{merely convex} cases where such assumptions do not hold.

For purposes of presentation, we work with the following version of \eqref{eq: bg} in which the reference point is taken to be a given initial iterate $\bx_0$ (similar to \cite{neu2024dealing}):
 \begin{equation}\label{eq: bg0}\tag{BG-0}
    \mathbb{E}\| \widetilde{\nabla} f(\bx) \|^2 \leq B^2\|\bx-\bx_0\|^2 + G^2.
\end{equation}
No generality is lost, since \eqref{eq: bg} and \eqref{eq: bg0} are equivalent for appropriate definitions of $B$, $G$, etc.
Specifically, if \eqref{eq: bg} holds, then \eqref{eq: bg0} holds with $B^2 = 2 \bar{B}^2$ and $G^2 = \bar{G}^2 + 2 \bar{B}^2 \|\bx_0 \|^2$, while if \eqref{eq: bg0} holds, then \eqref{eq: bg} holds with $\bar{B}^2 = 2B^2$ and $\bar{G}^2 = G^2 + 2B^2 \|\bx_0\|^2$.

Algorithmic parameters in the sequel depend on the constant $B$ and not on $G$. The estimation of $B$ is feasible in many cases. For example when $f(x)=\frac{1}{N}\sum_{i=1}^N f_i(x)$ where each $f_i$ convex and $L_i$-smooth, we have $B=\sqrt{\frac{1}{2N}\sum_{i=1}^N L_i}$. See \cite[Section 5.2.3]{wright2022optimization} for calculating $B$ for linear least squares.

A major difficulty to analyze SGD under \eqref{eq: bg0} is that the terms involving $\|\bx_k - \bx^\star\|^2$ no longer telescope, since \eqref{eq: bg0} brings an error term of this form. The work of \citet{domke2023provable} showed how to go around this difficulty with a fixed horizon, but their technique does not apply when we wish to get \emph{anytime} rates without a horizon. It is also not clear how to extend the idea in this paper (which is also used in \cite{khaled2022better}) in more general cases such as min-max optimization.
A related approach is taken in \cite[Lemma~5.2]{zhao2022randomized} for a block coordinate method, which corresponds to a finite  number of component functions. 
The bound for the norm of iterates obtained in this lemma scales exponentially in $B^2/N$ where $N$ is the number of  coordinate blocks.

{A particularly relevant setting for problems without bounded variance (which we expand on \Cref{sec: main_minmax}) is the min-max optimization 
\begin{equation}\label{eq: minmax_temp1}
    \min_{\bx\in X}\max_{\by\in Y} \, \mathcal{L}(\bx, \by),
\end{equation}
where $X\subseteq\mathbb{R}^d$, $Y\subseteq\mathbb{R}^n$ are closed and convex sets
and $\mathcal{L}$ is convex in $\bx$, and concave in $\by$.  
We consider too the special case in which the function $\mathcal{L}$  in \eqref{eq: minmax_temp1} is the Lagrangian for an algebraically constrained convex optimization problem
\begin{equation*}
    \min_{\bx} f(\bx) \quad  \text{~subject to~}\quad g_i(\bx) \leq 0, \quad i=1,2,\dotsc,n,
\end{equation*}
where $\bx \in \R^d$ and $f$ and $g_i$, $i=1,2,\dotsc,n$ are convex functions from $\R^d$ to $\R$. }
The paper of \cite{neu2024dealing} {analyzes} a stochastic algorithm for \eqref{eq: minmax_temp1}, making use of the assumption  \eqref{eq: bg0} {to get guarantees on the primal-dual gap (see \Cref{sec: main_minmax})}.
It presented a key insight that connects this classical assumption to the Halpern iteration \citep{halpern1967fixed} in a surprising way, providing the main motivation for our work. 
We expand on the approach in \cite{neu2024dealing} to provide an alternative perspective, then extend it to different settings in the sequel by distilling the simple and powerful idea. 
Our approach results in  elementary proofs for results that extend the state of the art in stochastic optimization without bounded-variance assumption.

\subsection{Contributions} 
For minimization problems, we have described the relationship between the classical assumption \eqref{eq: bg} with more recently proposed conditions relaxing bounded variance assumptions. 
Next, {in \Cref{sec: main_min},} building on the idea of \cite{neu2024dealing}, we establish anytime rate guarantees under \eqref{eq: bg0} of stochastic Halpern iteration  with variable parameters, then show last-iterate rates.

{In \Cref{sec: main_minmax},} we focus on min-max problems  under \eqref{eq: bg0} {since many classical examples of this template fits naturally to the setting of this assumption}.
We work with two optimality measures whose well definedness does not require boundedness of the feasibility set.
{In \Cref{subsec: nonlinprog},}  for convex optimization with convex inequality constraints, we extend the stochastic gradient descent-ascent algorithm from \cite{neu2024dealing} to handle variable parameters, then use convex duality to derive convergence rates for objective suboptimality and feasibility. 
{In \Cref{sec: opnorm_main},} we focus on \emph{residual} guarantees for min-max problems (a generalization of the \emph{gradient norm}-type guarantees often used in minimization problems) and analyze a variance-reduced extragradient algorithm with Halpern anchoring to obtain best-known convergence rates, all using \eqref{eq: bg} in place of the restrictive bounded-variance assumption, {but introducing an additonal mean-square smoothness condition.}

\section{Preliminaries}
We denote the Euclidean norm as $\|\cdot\|$. We define the projection onto a set $X$ as $P_X(\bx) = \arg\min_{\by\in X} \| \by-\bx\|^2$.
The indicator function $\delta_X$ is defined by $\delta_X(\bx)=0$ if $\bx\in X$ and $\delta_X(\bx)=\infty$ otherwise. 
Distance of a point $\bx$ to a set $X$ is denoted as $\dist(\bx, X) = \min_{\by\in X} \|\bx-\by\|$. An operator $T\colon\mathbb{R}^d \to\mathbb{R}^d$
 is {\em nonexpansive} if $\|T\bx-T\by\|\leq\|\bx-\by\|$.
 The notation $\mathbb{E}_k$ describes the expectation conditioned on the $\sigma$-algebra generated by the randomness of $\bx_k, \dots, \bx_1$.

\subsection{Halpern Iteration}
Originally proposed in \cite{halpern1967fixed} for finding fixed points of nonexpansive maps $T\colon\mathbb{R}^d\to\mathbb{R}^d$, the Halpern iteration is defined by
\begin{equation*}
    \bx_{k+1} = \beta_k \bx_0 + (1-\beta_k) T\bx_k,
\end{equation*}
for a $\beta_k \to 0$ satisfying certain requirements. 
Historically, an important property of Halpern iteration is that its iteration sequence $\{ \bx_k \}$ converges to a specific point in the solution set $X^\star$, namely, $P_{X^\star}(\bx_0)$. 
Another important property of Halpern iteration for infinite-dimensional Hilbert and Banach spaces is that its iterates converge \emph{strongly} \cite{xu2004viscosity, bauschke2017convex}.

Halpern iteration recently garnered interest for another property that emerges
when it is applied to min-max optimization. 
Consider $\min_{\bx\in\mathbb{R}^d} \max_{\by\in\mathbb{R}^n} \Lcal(\bx, \by)$ with smooth and convex-concave $\Lcal$.
As shown in \cite{diakonikolas2020halpern,yoon2021accelerated}, incorporating Halpern's idea of anchoring towards $\bx_0$ results in optimal guarantees for the last iterate (for making the norm of gradient of $\Lcal$ small), a behavior not achieved for classical min-max algorithms, such as extragradient \cite{korpelevich1976extragradient}. 
For these results, it is critical that the choice of $\beta_k$ is iteration-dependent, specifically, $\beta_k = \frac{1}{k+2}$. Investigation of this \emph{acceleration} behavior is an active area of research \cite{park2022exact,lee2021fast,yoon2022accelerated,cai2022accelerateda,cai2023accelerated,tran2024halpern}.

\subsection{Halpern meets Gladyshev for Stochastic Optimization}
A surprising connection between the Halpern iteration and assumption \eqref{eq: bg0} is due to \citet{neu2024dealing}, who showed that by choosing
\begin{equation*}
    \beta_k \equiv \beta= \frac{1}{K} \text{~~~and~~~} \tau_k \equiv \tau=\frac{1}{B\sqrt{K}},
\end{equation*}
for a given last iteration counter (horizon) $K > 0$, the algorithm
\begin{equation*}
    \bx_{k+1} = P_X(\beta \bx_0 +(1-\beta)\bx_k - \tau \widetilde{\nabla} f(\bx_k)),
\end{equation*}
has the optimal $O(1/\sqrt{K})$ rate on the objective for stochastic convex optimization under \eqref{eq: bg0}. 
This is a special case of the result for constrained stochastic min-max problems in \cite{neu2024dealing}.
The resemblance of this algorithm to Halpern iteration is clear, apart from a mismatch between the parameters used by \cite{neu2024dealing} and those typically used for Halpern iterations. 
In particular, as we pointed out above, having  $\beta_k \approx 1/k$ depend on iteration count $k$ is essential for Halpern-based min-max algorithms. In the sequel, we show that the main idea of \cite{neu2024dealing} still works when we define $\beta_k$ and $\tau_k$ to depend on $k$, showing that in fact the algorithm becomes precisely SGD with Halpern anchoring in view of \cite{yoon2021accelerated}.

\section{Convergence of Halpern Iteration for Minimization Problems under \eqref{eq: bg0}}\label{sec: main_min}

 In this section, we describe the convergence of the stochastic Halpern anchoring for convex optimization problems {\eqref{eq:proto}} under the assumption \eqref{eq: bg0}. 
 Specifically, we assume the following.
\begin{assumption}\label{asp: 1}
    Let $f\colon\mathbb{R}^d\to\mathbb{R}$ be convex, $X\subseteq \mathbb{R}^d$ be convex and closed. 
    Let the (potentially stochastic) oracle $\widetilde{\nabla} f$ satisfy \eqref{eq: unbiased} and \eqref{eq: bg0}, for a given initial point $\bx_0$.
\end{assumption}
Given $\bx_0$, the projected Halpern iteration for $k\geq 0$ is as follows.
\begin{equation}\label{alg: halpern1}
    \bx_{k+1} = P_X(\beta_k\bx_0 + (1-\beta_k)\bx_k - \tau_k \widetilde{\nabla}f(\bx_k)).
\end{equation}
Special cases of our results give rate guarantees for deterministic nonsmooth optimization with possibly non-Lipschitz $f$, which is also an active line of research, see for example \citep{grimmer2019convergence, zhao2022randomized}.

{
    Extending the developments in this section to the additively composite problem given as $\min_{\bx} f(\bx) + g(\bx)$ by using the proximal operator of $g$ is straightforward. For brevity, we omit this extension.}
\subsection{Single-iteration analysis}
The following lemma extends the idea of \cite{neu2024dealing} to allow variable parameters $\beta_k, \tau_k$, thus dispensing with the need to choose a fixed finite horizon $K$ for the number of iterations.
The proof of this and later results makes use of several auxiliary results proved in Section~\ref{sec: opnorm_proof}.
\begin{lemma}\label{lem: one_iteration_subg}
    Let \Cref{asp: 1} hold and $\{ \bx_k \}$ be generated by \eqref{alg: halpern1} with $\beta_k \in (0, 1/2]$ and $\tau_k \leq \frac{\sqrt{\beta_k(1-\beta_k)}}{\sqrt{6}B}$. Then for any $\bx\in X$ that is deterministic under conditioning of $\mathbb{E}_k$, 
    we have
\begin{align*}
        2\tau_k(f(\bx_k) - f(\bx)) + \Eb_k \| \bx_{k+1} - \bx\|^2 &\leq (1-\beta_k) \| \bx_k - \bx \|^2 + \beta_k \| \bx_0 - \bx \|^2 + \frac{\beta_k G^2}{3B^2}.
    \end{align*}
    
\end{lemma}
\begin{remark}
    If $\beta_k$ and $\tau_k$ were constants, our proof would almost mirror that of \cite{neu2024dealing} who took an  online learning perspective. 
    Our analysis is inspired by classical analyses of the Halpern iteration \citep{xu2004viscosity, bauschke2017convex}. 
    We discuss the differences after the proof.
\end{remark}
\begin{proof}[Proof of \Cref{lem: one_iteration_subg}]
By definitions of $\bx_{k+1}$ and projection, we have for any $\bx\in X$ that
\begin{align}
    2\langle \bx_{k+1} - (\beta_k \bx_0+(1-\beta_k)\bx_k)+\tau_k\widetilde{\nabla} f(\bx_k), \bx-\bx_{k+1}  \rangle \geq 0.\label{eq: bgf4}
\end{align}
We use \Cref{fact: beta} with $\ba\leftarrow\bx_{k+1}$, $\bar\msfx_k \leftarrow (1-\beta_k)\bx_k + \beta_k \bx_0$ and $\bb \leftarrow \bx$ to get
\begin{align}
2\langle \bx_{k+1} - (\beta_k \bx_0+(1-\beta_k) \bx_k), \bx-\bx_{k+1} \rangle &= -\| \bx-\bx_{k+1}\|^2 +(1-\beta_k) \| \bx-\bx_k\|^2 + \beta_k \| \bx-\bx_0\|^2 \notag \\
    &\quad- \beta_k \| \bx_{k+1} - \bx_0\|^2 -(1-\beta_k) \| \bx_{k+1} - \bx_k\|^2.\label{eq: bmn5}
\end{align}
For the remaining part of \eqref{eq: bgf4}, we take conditional expectation and use \eqref{eq: unbiased} to estimate
\begin{align}
    2\tau_k\Eb_k\langle \widetilde{\nabla} f(\bx_k), \bx-\bx_{k+1} \rangle    &= 2\tau_k\langle \Eb_k[\widetilde{\nabla} f(\bx_k)], \bx-\bx_{k} \rangle +2\tau_k\Eb_k \langle \widetilde{\nabla} f(\bx_k), \bx_k-\bx_{k+1} \rangle\notag \\
    &\leq 2\tau_k [f(\bx)-  f(\bx_k)] + \frac{2\tau_k^2}{1-\beta_k}\Eb_k\|\widetilde{\nabla} f(\bx_k)\|^2 + \frac{1-\beta_k}{2} \Eb_k\| \bx_{k+1} - \bx_k\|^2,\label{eq: bmn6}
\end{align}
where the last step used convexity for the first term and Young's inequality for the second.

By substituting \eqref{eq: bmn5} and \eqref{eq: bmn6} into \eqref{eq: bgf4}, taking conditional expectation, and rearranging, we obtain
\begin{align}
    &2\tau_k[f(\bx_k) - f(\bx)] + \Eb_k\|\bx-\bx_{k+1}\|^2 \notag \\
    &\leq (1-\beta_k) \|\bx-\bx_k\|^2 + \beta_k \| \bx-\bx_0\|^2\notag \\
    &\quad +\frac{2\tau_k^2}{1-\beta_k}\Eb_k\|\widetilde{\nabla} f(\bx_k)\|^2  - \frac{1-\beta_k}{2}\Eb_k\|\bx_{k+1} - \bx_k\|^2 - \beta_k\Eb_k\| \bx_{k+1} - \bx_0\|^2.\label{eq: gfk4}
\end{align}
Use of \eqref{eq: bg0} and Young's inequality results in
\begin{align}
    \Eb_k \| \widetilde{\nabla} f(\bx_k)\|^2 &\leq  B^2 \| \bx_k-\bx_0\|^2 + G^2 \leq \frac{3B^2}{2}\| \bx_{k}-\bx_{k+1}\|^2 + 3B^2\|\bx_{k+1} - \bx_0\|^2 + G^2.\label{eq: biu5}
\end{align}
By substituting the bound \eqref{eq: biu5} into \eqref{eq: gfk4} and using the definitions $\tau_k, \beta_k$ to argue that
\begin{equation*}
    \frac{6B^2\tau_k^2}{1-\beta_k} \leq \beta_k \text{~and~} \frac{3B^2\tau_k^2}{1-\beta_k}\leq \frac{1-\beta_k}{2},
\end{equation*}
we conclude that the last line of \eqref{eq: gfk4} is bounded by $\frac{2\tau_k^2G^2}{1-\beta_k} {\leq} \frac{\beta_k G^2}{3B^2}$, completing the proof.
 \end{proof}
 
 The main insight of \citet{neu2024dealing}, which we also rely on in this proof, is that one can use the \emph{good term} $-\beta_k\|\bx_{k+1} - \bx_0\|^2$ in \eqref{eq: gfk4} to cancel the contributions coming from the norm of $\bx_k$ in assumption \eqref{eq: bg0}, that is, the \emph{bad term} in the middle of the right-hand side of \eqref{eq: biu5}. 
 As we see here, this idea still works with definitions of  $\tau_k, \beta_k$ that depend on $k$. 
Because of our choice of $\beta_k$,  the algorithm we analyze has precisely the Halpern-based anchoring with no fixed horizon, see \cite{yoon2021accelerated}.

The analysis of \cite{neu2024dealing} reduces the original problem to a regularized problem and then deploys the regret analysis of mirror descent from \cite{duchi2010composite}, which uses constant step sizes since it bounds the uniform average of regret. 
By contrast, our analysis can be seen as working with the {\em weighted average} of regret, the weighting being done with dynamic step sizes (a trick also often used with SGD, see \cite{orabona2020last}).
Another difference between the analyses is that due to the reduction, \cite{neu2024dealing} uses convexity of the regularization term $\frac{\beta_k}{2} \| \bx-\bx_0\|^2$ whereas we use a direct expansion of the quadratic, leading to a \emph{tighter} estimate, with an additional negative term $-\beta_k\|\bx-\bx_k\|^2$ on the right-hand side. This term matters for
ensuring that there is no superfluous logarithmic term in the convergence rate in \Cref{cor: weighted_avg_rate}.

\subsection{Convergence rate for the weighted average}

The following corollary shows that for a weighted average of the iterates with higher weights on the later iterates, we have a rate $O(1/\sqrt{k})$ under \eqref{eq: bg0} with dynamic step sizes.
Thus, there is no need to set a horizon $K$ as a parameter in the algorithm, unlike the 
related results \cite[Thm. 8]{domke2023provable} and \cite{neu2024dealing}.
\begin{corollary}\label{cor: weighted_avg_rate}
    Let \Cref{asp: 1} hold. Let $\{\bx_k\}$ be generated by \eqref{alg: halpern1} with $\beta_k=\frac{1}{k+2}$ and $\tau_k = \frac{\sqrt{k+1}}{\sqrt{6}B(k+2)}$. Then for any $k \ge 1$, we have for {
    \[
    \bx_k^{\mathsf{out}} = \frac{1}{\sum_{i=0}^{k} (i+2) \tau_i} \sum_{i=0}^{k} (i+2) \tau_i \bx_i = \frac{1}{\sum_{i=0}^{k} \sqrt{i+1} } \sum_{i=0}^{k}\sqrt{i+1} \, \bx_i
    \]
    } that
    {
    \begin{equation*}
        \Eb\left[f(\bx_k^{\mathsf{out}})-f(\bx^\star)\right] \leq \frac{1}{\sqrt{k+1}}\left(3B \|\bx_0-\bx^\star\|^2 +  \frac{2G^2}{3B}\right).
    \end{equation*}
    }
\end{corollary}
\begin{remark}
    It is worth emphasizing that even with variable step sizes, the convergence
    rate does not suffer from superfluous log terms, unlike the often case
    with SGD without bounded domains and variable step size (see, e.g. \citep[Thm.~5.7]{garrigos2023handbook}). 
 Weighted averaging allows the elimination of such terms.
\end{remark}
\begin{proof}[Proof of \Cref{cor: weighted_avg_rate}]
We start from the inequality in \Cref{lem: one_iteration_subg}, substitute $\bx=\bx^\star$ {and the definitions of $\beta_k$ and $\tau_k$, multiply both sides by $k+2$ (from which $(k+2) \tau_k = \sqrt{k+1}/(\sqrt{6} B) $),}
and take total expectation, to arrive at
\begin{align}
\frac{2}{\sqrt{6} B} \sqrt{k+1} \,  \Eb[f(\bx_k)-f(\bx^\star)]&\leq  (k+1)\Eb\|\bx_k-\bx^\star\|^2 -(k+2)\Eb\|\bx_{k+1} - \bx^\star\|^2 \notag \\&\quad + \|\bx_0-\bx^\star\|^2
+\frac{G^2}{3B^2}.\label{eq: lrt4}
\end{align}
{By using convexity of $f$ and summing \eqref{eq: lrt4} over $k=0,1,\dotsc,K$ (for any $K \ge 1$), we obtain
\begin{align*}
\frac{2}{\sqrt{6}B} \left( \sum_{k=0}^K \sqrt{k+1} \right) \mathbb{E} ( f(\bx_K^{\text{out}}) -f(\bx^\star))  & \le 
    \frac{2}{\sqrt{6}B}   \sum_{k=0}^K \sqrt{k+1} \, \mathbb{E} (f(\bx_k)-f(\bx^\star)) \\
    &  \le (K+2) \| \bx_0-\bx^\star\|^2 + \frac{G^2}{3 B^2}(K+1).
\end{align*} 
By using the lower bound
\[
\sum_{k=0}^K \sqrt{k+1} \geq  \frac{2}{3} (K+1)^{3/2},
\]
we obtain
\begin{align*}
\mathbb{E} \left[ f(\bx_K^{\text{out}}) - f(\bx^*) \right] &\le \frac{3\sqrt 6 B (K+2)}{4 (K+1)^{3/2}} \| \bx_0-\bx^\star\|^2 + \frac{\sqrt 6 G^2}{4B \sqrt{K+1}} \\ & \le \frac{3B}{\sqrt{K+1}} \| \bx_0-\bx^*\|^2 + \frac{2 G^2}{3B \sqrt{K+1}},
\end{align*}
where we used that $\frac{3\sqrt{6}}{4}(K+2)\le  3 (K+1)$ for $K\ge 1$, and $\frac{\sqrt{6}}{4}\leq \frac 2 3$.}
\end{proof}

\subsection{Convergence rate for the last iterate}
As mentioned above, one reason for renewed interest in the  Halpern iteration is that it allows optimal last-iterate guarantees for deterministic min-max optimization. 
We show that SGD with Halpern anchoring also achieves last-iterate guarantees for stochastic optimization under \eqref{eq: bg0}, by adapting the ideas from \cite{orabona2020last,shamir2013stochastic} for the Halpern iteration and the BG assumption.
\begin{theorem}\label{th: last_it_sgd}
    Let \Cref{asp: 1} hold. Let $\{\bx_k\}$ be generated by \eqref{alg: halpern1} with $\beta_k=\frac{1}{k+2}$ and $\tau_k= \frac{\sqrt{k+1}}{\sqrt{6}B(k+2)}$. Then for any $k\geq0$ we have 
    \begin{align*}
        \Eb[f(\bx_k)] - f(\bx^\star) &\leq {\frac{2\sqrt{6}B(1+\log(k+2))}{\sqrt{k+1}}\left(3\|\bx_0-\bx^\star\|^2 + \frac{G^2}{B^2} \right)}= \widetilde{O}\left(\frac{1}{\sqrt{k}}\right).
    \end{align*}
\end{theorem}
\begin{proof}
Let us set $l\in \{ 1,\dots, K-1\}$ for some $K>0$. We take the result of \Cref{lem: one_iteration_subg} for $k=K-l,\dots, K$ take conditional expectation and sum:
\begin{align*}
    \sum_{k=K-l}^K 2\tau_k\Eb_{K-l}[f(\bx_k) - f(\bx)] &\leq (1-\beta_{K-l}) \|\bx_{K-l}-\bx\|^2 + \sum_{k=K-l}^K\Big(\beta_k \| \bx_0-\bx\|^2 + \frac{\beta_kG^2}{3B^2}\Big).
\end{align*}
We plug in $\bx=\bx_{K-l}$ (which is permitted as per the requirement in \Cref{lem: one_iteration_subg} since we use the inequality for $k\geq K-l$) and take total expectation to obtain
\begin{align}\label{eq: bnm5}
    \sum_{k=K-l}^K 2\tau_k\Eb[f(\bx_k) - f(\bx_{K-l})] &\leq \sum_{k=K-l}^K \Big(\beta_k \Eb\| \bx_{K-l}-\bx_0\|^2 +\frac{\beta_kG^2}{3B^2}\Big).
\end{align}
We now estimate like \cite[Lemma 1]{orabona2020last} (see also \cite[Lemma 17]{lin2016iterative}). We have
\begin{align}
    \sum_{k=K-l}^K \tau_k(f(\bx_k)-f(\bx_{K-l})) &= \sum_{k=K-l}^K \tau_k[f(\bx_k) - f(\bx^\star)-f(\bx_{K-l})+f(\bx^\star)] \notag \\
    &\geq \sum_{k=K-l}^K \Bigl(\tau_k[f(\bx_k) - f(\bx^\star)]-\tau_{K-l}[f(\bx_{K-l})-f(\bx^\star)]\Bigr),\label{eq: gor5}
\end{align}
because $\tau_{K-l}\geq \tau_{k}$ for $k\geq K-l$ and $f(\bx_{K-l})-f(\bx^\star)\geq 0$. 
Let us define 
\begin{align}\label{eq: def_sl}
    S_l = \frac{1}{l} \sum_{k=K-l+1}^K \tau_k (f(\bx_k)-f(\bx^\star)),
\end{align}
which immediately implies
\begin{align}\label{eq: gbn5}
    lS_l = (l+1)S_{l+1} - \tau_{K-l} [f(\bx_{K-l})-f(\bx^\star)] 
    \iff S_l =  S_{l+1} + \frac{1}{l} \left( S_{l+1} - \tau_{K-l} [f(\bx_{K-l})-f(\bx^\star)]\right).
\end{align}
Using \eqref{eq: def_sl} in \eqref{eq: gor5} and using that the first term in the right-hand side of \eqref{eq: gor5} is $(l+1)S_{l+1}$ give
\begin{align*}
    \sum_{k=K-l}^K \tau_k(f(\bx_k)-f(\bx_{K-l})) \geq (l+1) S_{l+1} - (l+1)\tau_{K-l}[f(\bx_{K-l}-f(\bx^\star))]
\end{align*}
and consequently (after dividing both sides by $l+1$):
\begin{align*}
    S_{l+1} -\tau_{K-l}[f(\bx_{K-l}-f(\bx^\star))] \leq \frac{1}{l+1}\sum_{k=K-l}^K \tau_k(f(\bx_k)-f(\bx_{K-l})).
\end{align*}
Plugging this into \eqref{eq: gbn5} gives
\begin{equation}
    S_l \leq S_{l+1} + \frac{1}{l(l+1)}\sum_{k=K-l}^K \tau_k(f(\bx_k)-f(\bx_{K-l})).
\end{equation}
Taking expectation, using \eqref{eq: bnm5} to bound the last term on the right-hand side and summing the resulting inequality for $l=1,\dots,K-1$ gives 
\begin{align}
    \tau_K\Eb[f(\bx_K) - F(\bx^\star)] &= \Eb[S_1] \notag \\
    &\leq \Eb[S_{K}] + \sum_{l=1}^{K-1} \frac{1}{2l(l+1)}\Big( \Eb\sum_{k=K-l}^K \beta_k \| \bx_{K-l}-\bx_0\|^2 +\sum_{k=K-l}^K \frac{\beta_kG^2}{3B^2} \Big).\label{eq: sgn4}
\end{align}
First, by substituting $\bx^\star$ in \Cref{lem: one_iteration_subg}, taking total expectation, and summing the resulting inequality for $k= 0, \dots, K$, using $\tau_0[f(\bx_0) - f(\bx^\star)] \geq 0$, and dividing by $2K$, we have 
\begin{align}\label{eq: axz4}
    \Eb[S_{K}] &= \frac{1}{K} \sum_{k=1}^K \tau_k\Eb[f(\bx_k)-f(\bx^\star)] \notag \\
    &\leq\frac{1+\sum_{k=0}^K\beta_k}{K}\Big(\frac{1}{2}\|\bx_0-\bx^\star\|^2 + \frac{G^2}{6B^2}\Big) \notag \\
    & \le   \frac{1+\log(K+2)}{K}\left( \frac{1}{2}\| \bx_0- \bx^\star\|^2+\frac{G^2}{6B^2} \right),
  \end{align}
  because $\beta_k = \frac{1}{k+2}\leq 1/2$ for $k\geq 0$.
Second, we lower bound the left-hand side of \eqref{eq: lrt4} by $0$ and then sum for $k=0,\dots K-1$ and divide by $K+1$ to get
\begin{equation*}
    \Eb\|\bx_K - \bx^\star\|^2 \leq \|\bx_0-\bx^\star \|^2 + \frac{G^2}{3B^2}
\end{equation*}
and hence for any $l=1,\dots, K-1$, we have that 
\begin{equation}\label{eq: axz5}
    \mathbb{E}\|\bx_{K-l}-\bx_0\|^2 \leq 2\mathbb{E}\|\bx_{K-l}-\bx^\star\|^2 + 2\mathbb{E}\|\bx^\star-\bx_0\|^2 \leq 4\|\bx_0-\bx^\star\|^2 + \frac{2G^2}{3B^2}.
\end{equation}
Next, since $\beta_k = \frac{1}{k+2}$ and $\tau_k^2=\Theta(\beta_k)$ by using precisely the same estimation as \cite[Corollary 3]{orabona2020last} (see also \cite[Lemma 17]{lin2016iterative}), we get
\begin{align}\label{eq: axz6}
    \sum_{l=1}^{K-1} \frac{1}{l(l+1)} \sum_{k=K-l}^K\beta_k &=  \frac{3(1+\log(K+1))}{K} .
\end{align}
Finally, we obtain the result by combining \eqref{eq: axz4}, \eqref{eq: axz5}, and \eqref{eq: axz6} in \eqref{eq: sgn4} and dividing both sides by $\tau_K$ and simplifying constants.
\end{proof}
The previous analyses relaxing bounded-variance assumptions for stochastic optimization did not have guarantees in the last iterate \citep{garrigos2023handbook,khaled2023unified}. 
Our result illustrates the flexibility of \eqref{eq: bg0} and the idea of \cite{neu2024dealing} to accommodate this assumption into last-iterate analyses of stochatic gradient methods.

It is worth discussing \Cref{th: last_it_sgd} in the context  of the Halpern-based algorithms that have gained traction recently for min-max problems \citep{diakonikolas2020halpern,yoon2021accelerated}. 
One of the main features of the latter algorithms in the deterministic case is that they have the optimal rate and complexity guarantees for the last iterate, when progress is measured with the appropriate extension of \emph{gradient norm}. 
For stochastic min-max problems, results to date for Halpern-based methods have shown only suboptimal last-iterate guarantees, and increasing mini-batch sizes is essential for existing analyses \citep{lee2021fast,cai2022stochastic}. 
The mechanism behind the last-iterate guarantees for these analyses is distinct from the mechanism behind the proof of \Cref{th: last_it_sgd}, which adapts the last-iterate convergence proof often used for SGD \citep{orabona2020last,shamir2013stochastic,zhang2004solving}. We do not need mini-batch sizes to increase during the run.

{A key} difference is that the analyses in \cite{lee2021fast} show a guarantee for the gradient norm whereas we show a guarantee for objective suboptimality; {these results are essentially complementary. 
We argue that the previous approaches for analyzing Halpern-based methods do not work well for obtaining guarantees on function suboptimality.}
A unification of these various ways of analyzing the Halpern iteration is a subject for future research.

\section{Primal-Dual Algorithms for Min-Max Optimization and Constrained Convex Optimization}\label{sec: main_minmax}

In this section, we {consider} the min-max optimization template
\begin{equation}\label{eq: minmax_temp}
    \min_{\bx\in X}\max_{\by\in Y} \, \mathcal{L}(\bx, \by),
\end{equation}
where $X\subseteq\mathbb{R}^d$, $Y\subseteq\mathbb{R}^n$ are closed and convex sets
and $\mathcal{L}$ is convex in $\bx$, and concave in $\by$. 
{We focus on min-max problems because the bounded variance assumption is even more restrictive for these problems than it is for minimization, making it a natural setting to investigate the use of the  BG assumption. Consider, for example, the  linearly constrained optimization problem
\begin{align*}
    \min_{\bx} f(\bx) \text{~subject to~} A\bx=\bb,
\end{align*}
with its equivalent min-max formulation
\begin{align*}
    \min_{\bx}\max_{\by} f(\bx) +\langle A\bx -\bb, \by \rangle.
\end{align*}
Suppose the stochastic gradients have the form
\begin{align*}
    \widetilde{\nabla}_{\bx} \mathcal{L}(\bx, \by) &= \nabla f(\bx) + \widetilde{A}^\top \by,\\
    \widetilde{\nabla}_{\by} \mathcal{L}(\bx, \by) &= \widetilde{A} \bx-\bb,
\end{align*}
where $\widetilde{A}$ is a random variable with $\mathbb{E}[\widetilde{A}] = A$. 
(For example, $\widetilde{A}$ could be a weighted single row or column of $A$.) Here, the norms of the stochastic gradients scale linearly with $\bx$ and $\by$, so the bounded stochastic gradient or bounded variance assumptions do not hold, in general, since the domains of $\bx$ and $\by$ are not  bounded. 
}

{Motivated by this observation,  we focus on particular cases of min-max problems in which the BG assumption allows us to solve problems with unbounded domains and unbounded variance.}
We start by discussing different optimality measures for this problem template.

\subsection{Optimality Measures}
A standard way to certify optimality for min-max problem is the \emph{duality gap} (see, e.g., \cite{facchinei2003finite}), defined as
\begin{equation*}
    \Gap(\bx_k, \by_k) = \max_{(\bx, \by)\in X\times Y} \left( \mathcal{L}(\bx_k, \by) - \mathcal{L}(\bx, \by_k) \right).
\end{equation*}
It is easy to see (by setting, e.g., $X=\mathbb{R}$, $Y=\mathbb{R}$, $\mathcal{L}(x, y) = xy$) that the duality gap can be infinite when $X$ and $Y$ are unbounded.
In this, a commonly used relaxation is the \emph{restricted duality gap} (see, e.g., \cite{nesterov2007dual}) which is defined by choosing bounded sets $\bar X$, $\bar Y$ and defining
\begin{equation*}
    \Gap_{(\bar X, \bar Y)}(\bx_k, \by_k) = \max_{(\bx, \by)\in \bar X\times \bar Y} \left( \mathcal{L}(\bx_k, \by) - \mathcal{L}(\bx, \by_k) \right).
\end{equation*}
This was the optimality measure used in \cite{neu2024dealing} while analyzing stochastic gradient descent-ascent under \eqref{eq: bg0}.
For the restricted duality gap to be a valid optimality measure --- that is,  for it to be $0$ if and only if $(\bx_k,\by_k)=(\bx^\star, \by^\star)$ --- the sets $\bar X, \bar Y$ must contain $\bx^\star, \by^\star$ and the whole trajectory of the algorithm; see \cite{nesterov2007dual}. 
This requirement is especially difficult to guarantee in a stochastic optimization setting,  since often the iterates of these algorithms cannot be proven  to stay in a uniformly bounded set. 
We thus have a contradiction, since the main motivation for using \eqref{eq: bg} is to deal with unboundedness in $X$ and  $Y$.
To address this issue, we consider two optimality measures that will not have the drawbacks of duality gap. 

Consider first the special case of convex nonlinear programming:
\begin{equation}\label{eq: nlp}
    \min_{\bx} f(\bx) \quad  \text{~subject to~}\quad g_i(\bx) \leq 0, \quad i=1,2,\dotsc,n.
\end{equation}
where $\bx \in \R^d$ and $f$ and $g_i$, $i=1,2,\dotsc,n$ are convex functions from $\R^d$ to $\R$. 
In this  classical case, it is critical to handle unbounded domains, since the size of the dual domain depends on the set of dual solutions, which we do not know in advance.
In this case, a natural measure of optimality is objective suboptimality and feasibility:
\begin{equation} \label{eq:uh1}
    |f(\bx_k) - f(\bx^\star) |\text{~~~and~~~} \sum_{i=1}^n \max(0, g_i(\bx_k)).
\end{equation}
{\Cref{subsec: nonlinprog} will focus on deriving guarantees for these quantities.}

A second optimality measure is applicable for the general case \eqref{eq: minmax_temp} {(for problems that may not fit the template \eqref{eq: nlp})}, where $\mathcal{L}(\bx, \by) = \Psi(\bx, \by) + h_1(\bx) - h_2(\by)$ with smooth $\Psi$, convex and nonsmooth $h_1, h_2$, and $X=\mathbb{R}^d, Y=\mathbb{R}^n$.
Optimality conditions of \eqref{eq: minmax_temp} can then be summarized as
\begin{equation*}
    0\in \binom{\nabla_\bx \Psi(\bx^\star, \by^\star) + \partial h_1(\bx^\star)}{-\nabla_\by \Psi(\bx^\star, \by^\star) + \partial h_2(\by^\star)}.
\end{equation*}
We also consider the \emph{residual}, sometimes also referred to as the \emph{tangent residual} \citep{cai2023accelerated},  a generalization of gradient norm from optimization, defined by 
$\res_k=\dist\left(0, \binom{\nabla_\bx \Psi(\bx_k, \by_k) + \partial h_1(\bx_k)}{-\nabla_\by \Psi(\bx_k, \by_k) + \partial h_2(\by_k)}\right)$.
{\Cref{sec: opnorm_main} will focus on deriving guarantees for this quantity.}

\subsection{Constrained Convex Optimization}\label{subsec: nonlinprog}
For the {convex nonlinear programming problem of} \eqref{eq: nlp}, the Lagrangian is defined by
\begin{equation}\label{eq: nlp_ldef}
    \mathcal{L}(\bx, \by) = f(\bx) + \sum_{i=1}^n  y_i g_i(\bx), \text{~with~} X=\mathbb{R}^d \text{~and~} Y=\mathbb{R}^n_+.
\end{equation}
In the sequel, we denote by $y_{k,i}$ the $i$-th coordinate of vector $\by_k$ and denote $\bz=\binom{\bx}{\by}$.
\begin{assumption}\label{asp: nlp}
The following conditions hold for \eqref{eq: minmax_temp} with $\mathcal{L}$ defined by \eqref{eq: nlp_ldef}.
{
\begin{enumerate}[itemsep=3pt]
\item The functions $f$ and $g_i$, $i=1,2,\dotsc,n$ are convex. There exists a primal-dual solution pair $(\bx^\star, \by^\star)$.
\item 
We have access to  oracles $\widetilde{\nabla}f(\bx)$, $\widetilde{\nabla}g_i(\bx)$, and $\widetilde{g_i}(\bx)$ such that
\begin{align*}
\Eb[\widetilde{\nabla}f(\bx)]\in\partial f(\bx), ~~ \Eb[\widetilde{\nabla}g_i(\bx)]\in\partial g_i(\bx),~~ \Eb[\widetilde{g_i}(\bx)] = g_i(\bx)~ \forall i=1,2,\dotsc,n.
\end{align*} 
\item The oracle $\widetilde{F}(\bz) = \binom{\widetilde{\nabla}_\bx L(\bx, \by)}{-\widetilde{\nabla}_\by L(\bx, \by)}$, 
defined as 
\begin{align*}
    \widetilde{\nabla}_\bx L (\bx, \by) = \widetilde{\nabla} f(\bx) + n y_i \widetilde{\nabla}g_i(\bx), \quad \widetilde{\nabla}_\by L(\bx, \by) = n\widetilde{g_i}(\bx), \text{~where~} i\sim\mathrm{Unif}\{1,\dots,n\},
\end{align*}
satisfies $
        \mathbb{E}\|\widetilde{F}(\bz)\|^2 \leq B^2 \| \bz-\bz_0\|^2 +G^2$ where $\mathbb{E}[\widetilde{F}(\bz)] = F(\bz)$ (cf. \eqref{eq: bg0}). 
\end{enumerate}
}
\end{assumption}
The last requirement in this assumption is satisfied
        as long as ${\bx\mapsto} \widetilde{g_i}(\bx)$ grows no faster than linear and $\widetilde{\nabla}g_i(\bx)$ is {bounded for all $\bx$}, {for $i=1,2,\dotsc,n$.}
{For example,} both conditions are satisfied when each $g_i$ is a Lipschitz continuous, convex function (without requiring Lipschitzness from $f$). {The oracle described in the second requirement of the assumption is commonly used for solving \eqref{eq: nlp_ldef} in the stochastic case, see \cite{lan2020algorithms} or \cite{yan2022adaptive}.}

Using the above notation for $\widetilde{F}$, we generalize the iteration \eqref{alg: halpern1} to solve the min-max problem \eqref{eq: minmax_temp} with $\mathcal{L}$ as in \eqref{eq: nlp_ldef}:
\begin{equation}\label{eq: fb_halp_nlp}
    \bz_{k+1} = P_Z(\beta_k \bz_0 + (1-\beta_k)\bz_k - \tau_k \widetilde{F}(\bz_k)),
\end{equation}
where $Z=X\times Y$ (see \eqref{eq: nlp_ldef}).
This is almost the same algorithm as the one in \cite{neu2024dealing}: a gradient descent-ascent method with Halpern anchoring, difference being the ability to use dynamic parameters $\beta_k, \tau_k$. 
We will show how the idea from \cite{neu2024dealing}, along with our extension by using dynamic parameters $\beta_k$ and $\tau_k$, can lead to anytime guarantees on objective suboptimality and feasibility by utilizing convex duality arguments, see for example \cite{yan2022adaptive}.

The following result shows convergence of an averaged-iterate sequence  according to expected values of the suboptimality-feasibility convergence measure \eqref{eq:uh1}.
This result makes use of two technical lemmas (Lemmas~\ref{lem: nlp1} and \ref{lem: nlp2}), whose statements and proofs appear after the statement of the proposition.

\begin{proposition}\label{cor: obj_feas}
    Let \Cref{asp: nlp} hold and $\{\bz_k\}$ be generated by \eqref{eq: fb_halp_nlp} with $\beta_k=\frac{1}{k+2}$, $\tau_k=\frac{1}{5B\sqrt{k+2}}$. Then, for $\xoutk := \frac{1}{\sum_{i=0}^{k-1}\tau_i}\sum_{i=0}^{k-1}\tau_i\bx_i$, we have
    {
    \begin{align*}
        \Eb|f(\xoutk) - f(\bx^\star)|&\leq \frac{1}{\sum_{i=0}^k\tau_i}\left(\frac{3+\log(k+1)}{2} D_{\star, 1}+\|\bx_0\|^2+\|\by_0\|^2+\sum_{i=0}^{k-1} 5\tau_i^2G^2 \right),\\
        &= \widetilde{O}\left(\frac{1}{\sqrt{k}}\right), \\
\sum_{i=1}^n \Eb[\max(0, g_i(\xoutk))]&\leq \frac{1}{\sum_{i=0}^k\tau_i}\left(\frac{3+\log(k+1)}{2} D_{\star, 2}+\|\bx_0\|^2+\|\by_0\|^2+\sum_{i=0}^{k-1} 5\tau_i^2G^2 \right),\\
        &= \widetilde{O}\left(\frac{1}{\sqrt{k}}\right),
    \end{align*}
    where 
    \begin{align*}
      D_{\star, 1} & = \|\bx^\star-\bx_0\|^2 + 2(\|\by^\star-\by_0\|^2+\|\by^\star\|^2)+\|\by_0\|^2, \\
      D_{\star, 2} & = \|\bx^\star - \bx_0\|^2 + \| \by^\star+1-\by_0\|^2+\|\by_0\|^2.
    \end{align*}}
  \end{proposition}
  This result is an alternative to \cite{yan2022adaptive}, where  a bounded primal domain was assumed. 
  Our result requires boundedness of neither primal nor dual domains.
  
The proof of this proposition relies on the following two lemmas, and the proof of the proposition appears at the end of this section.
  The first lemma extends \cite{neu2024dealing} to the case in which parameters $\tau_k$ and  $\beta_k$ are variable.
\begin{lemma}\label{lem: nlp1}
    Let \Cref{asp: nlp} hold and suppose that $\{ \bz_k\}$ be generated by \eqref{eq: fb_halp_nlp}. 
    For any $\bz\in Z$, we have for $K>0$ that
    \begin{align*}
        & \sum_{k=0}^{K-1}2\tau_k \langle F(\bz_k),\bz_k-\bz \rangle - (3+\log (K+1)) \|\bz - \bz_{0}\|^2\\
        &\leq \sum_{k=0}^{K-1}\Big( \frac{2\tau_k^2}{1-\beta_k}\|\widetilde{F}(\bz_k)\|^2 + 2\tau_k\langle F(\bz_k) - \widetilde{F}(\bz_k), \bz_k \rangle\Big)+  \bigg\|\sum_{k=0}^{K-1} \tau_k(F(\bz_k) - \widetilde{F}(\bz_k))\bigg\|^2 \\
        &\quad+2\|\bz_0\|^2-\sum_{k=0}^{K-1}\Big( \frac{1-\beta_k}{2} \|\bz_{k+1} - \bz_k\|^2 + \beta_k\|\bz_0-\bz_{k+1}\|^2 \Big).
    \end{align*}
\end{lemma}
\begin{proof}
    We simply follow the proof of \Cref{eq: bmn5} with $\widetilde{F}$ instead of $\widetilde{\nabla} f$. That is, by using \eqref{eq: bgf4} and \eqref{eq: bmn5} with this change, we have
    \begin{align}
        2\tau_k \langle {F}(\bz_k), \bz_k-\bz \rangle+ &\|\bz-\bz_{k+1}\|^2 \leq (1-\beta_k) \| \bz - \bz_k\|^2 + \beta_k\|\bz - \bz_0\|^2 \notag \\
        &\quad -(1-\beta_k) \|\bz_{k+1} - \bz_k\|^2 - \beta_k\|\bz_{0}-\bz_{k+1}\|^2\notag \\
        &\quad +2\tau_k \langle \widetilde{F}(\bz_k), \bz_k - \bz_{k+1} \rangle + 2\tau_k \langle F(\bz_k) - \widetilde{F}(\bz_k), \bz_k-\bz \rangle,\label{eq: snm4}
    \end{align}
    where we added to both sides $2\tau_k\langle F(\bz_k), \bz_k-\bz\rangle$ with {$F(\bz_k)=\mathbb{E}[\widetilde{F}(\bz_k)]$}.
    
    We use Young's inequality twice to obtain
    \begin{align*}
        2\tau_k \langle \widetilde{F}(\bz_k), \bz_k - \bz_{k+1}\rangle &\leq \frac{2\tau_k^2}{1-\beta_k} \| \widetilde{F}(\bz_k)\|^2 + \frac{1-\beta_k}{2} \| \bz_k - \bz_{k+1}\|^2,\\
    2\sum_{k=0}^{K-1} \tau_k\langle F(\bz_k) - \widetilde{F}(\bz_k), -\bz \rangle &\leq \bigg\|\sum_{k=0}^{K-1} \tau_k(F(\bz_k) - \widetilde{F}(\bz_k))\bigg\|^2 +\|\bz\|^2.
    \end{align*}
    Summing up \eqref{eq: snm4} for $k=0,\dots, K-1$, substituting the last two estimates and using Young's inequality as $\|\bz\|^2 \leq 2\| \bz-\bz_0\|^2 + 2\| \bz_0\|^2$ finish the proof.
\end{proof}

The next  lemma follows from convex duality arguments; see for example \cite[Lemma 9]{yan2022adaptive}. Note that the arguments in this lemma are deterministic.
\begin{lemma}\label{lem: nlp2}
    Let \Cref{asp: nlp} hold.
    If for any $\bz=(\bx, \by)\in Z$ and $K>0$ we have 
    \begin{equation}\label{eq: ddd4}
        \frac{1}{\sum_{i=0}^{K-1}\tau_i} \sum_{k=0}^{K-1} \tau_k\langle F(\bz_k), \bz_k - \bz \rangle - c_K \|\bz-\bz_0\|^2 \leq d_K,
    \end{equation}
    for some positive $c_K$ and  $d_K$, then it follows that for $\xoutK := \frac{1}{\sum_{i=0}^{K-1}\tau_i}\sum_{i=0}^{K-1}\tau_i\bx_i$, we have
    \begin{align*}
        |f(\xoutK) - f(\bx^\star)|&\leq c_K(\|\bx^\star-\bx_0\|^2 + 2(\|\by^\star-\by_0\|^2+\|\by^\star\|^2)+\|\by_0\|^2)+d_K,\\
        \sum_{i=1}^n \max(0, g_i(\xoutK)) &\leq c_K (\|\bx^\star - \bx_0\|^2 + \| \by^\star+1-\by_0\|^2+\|\by_0\|^2) + d_K.
    \end{align*}
  \end{lemma}
  \begin{proof}
      By the definition $F(\bz) = \binom{\mathbb{E}[\widetilde{\nabla}_\bx \mathcal{L}(\bx, \by)]}{-\mathbb{E}[\widetilde{\nabla}_\by \mathcal{L}(\bx, \by)]}\in\binom{\partial_\bx \Lcal(\bx, \by)}{-\partial_\by \Lcal(\bx, \by)}$ (see also \eqref{eq: nlp_ldef}), we have
\begin{align*}
    \langle F(\bz_k), \bz_k-\bz \rangle &= \langle \mathbb{E}[\widetilde{\nabla}_\bx \Lcal(\bx_k, \by_k)], \bx_k - \bx \rangle - \langle \mathbb{E}[\widetilde{\nabla}_\by \Lcal(\bx_k, \by_k)], \by_k - \by \rangle \\
    &= \langle \mathbb{E}[\widetilde{\nabla} f(\bx_k)] + \sum_{i=1}^n y_{k, i} \mathbb{E}[\widetilde{\nabla} g_i(\bx_k)], \bx_k-\bx \rangle - \sum_{i=1}^n g_i(\bx_k)(y_{k, i}-y_i).
\end{align*}
Due to \Cref{asp: nlp} and $y_{k, i}\geq0$, we have 
\begin{equation*}
     y_{k, i} \langle \mathbb{E}[\widetilde{\nabla} g_i(\bx_k)], \bx_k - \bx \rangle \geq y_{k, i}(g_i(\bx_k) - g_i(\bx)).
  \end{equation*}
Combining the last two estimates and convexity of $f$ and \Cref{asp: nlp} on \eqref{eq: ddd4} gives that
\begin{equation*}
        \frac{1}{\sum_{i=0}^{K-1}\tau_i} \sum_{k=0}^{K-1} \tau_k\Big(f(\bx_k) + \sum_{i=1}^n y_i  g_i(\bx_k) - f(\bx) - \sum_{i=1}^n y_{k, i} g_i(\bx) \Big) - c_K \|\bz-\bz_0\|^2 \leq d_K.
    \end{equation*}
By the definition of $\xoutK$, convexity of $f, g_i$, and {$\youtK = \frac{1}{\sum_{i=0}^{K-1}\tau_i}\sum_{k=0}^{K-1}\tau_k\by_k$}, we have
\begin{align*}
    f(\xoutK) + \sum_{i=1}^n y_i  g_i(\xoutK) - f(\bx) - \sum_{i=1}^n y_{K, i}^{\mathsf{out}} g_i(\bx)  - c_K\|\bz-\bz_0\|^2 \leq d_K.
\end{align*}
After setting $\bx=\bx^\star$ and using $y_{K, i}^{\mathsf{out}}\geq0$ and $g_i(\bx^\star) \leq 0$ gives
\begin{align}\label{eq: mtw2}
    f(\xoutK) + \sum_{i=1}^n y_i  g_i(\xoutK) - f(\bx^\star)  - c_K\|\by-\by_0\|^2 \leq c_K\|\bx^\star - \bx_0\|^2 + d_K.
\end{align}
By taking $y_i = 1+y_i^\star$ when $g_i(\bx_K^{\mathsf{out}}) > 0$ and $y_i = 0$ otherwise in \eqref{eq: mtw2}, one obtains
\begin{align}\label{eq: sfr6}
    &f(\xoutK) + \sum_{i=1}^n (1+y_i^\star)  \max(0, g_i(\xoutK)) - f(\bx^\star)\notag \\
    &\leq  c_K\left(\|\by^\star + 1-\by_0\|^2+\|\by_0\|^2 +\|\bx^\star - \bx_0\|^2\right) + d_K.
\end{align}
Since, under our assumptions, the primal-dual solution is the saddle point of the Lagrangian, we can write
$    \Lcal(\xoutK, \by^\star) \geq \Lcal(\bx^\star, \by^\star),$
which by the definition of $\Lcal$ gives
$    f(\bx_K^{\mathsf{out}}) - f(\bx^\star) \geq -\sum_{i=1}^n y_i^\star g_i(\bx_K^{\mathsf{out}}),$
  since $y_i^\star g_i(\bx^\star) = 0$.
Due to $y_i^\star \geq 0$, this implies 
\begin{equation}\label{eq: fbn4}
    f(\bx_K^{\mathsf{out}}) - f(\bx^\star) \geq -\sum_{i=1}^n y_i^\star \max(0, g_i(\bx_K^{\mathsf{out}})).
\end{equation}
Combining this with \eqref{eq: sfr6} gives the result on feasibility.

Using $y_i =2y_i^\star$ if $g_i(\bx_K^{\mathsf{out}}) > 0$; and $y_i=0$ otherwise in \eqref{eq: mtw2} gives
\begin{equation*}
    f(\bx_K^{\mathsf{out}}) - f(\bx^\star) + \sum_{i=1}^n 2y^\star_i \max(0, g_i(\bx_K^{\mathsf{out}})) \leq c_K(\|\bx^\star-\bx_0\|^2 + 2(\|\by^\star-\by_0\|^2+\|\by^\star\|^2)+\|\by_0\|^2)+d_K.
\end{equation*}
Using \eqref{eq: fbn4} on this estimate results in
\begin{equation}\label{eq: axq33}
    -(f(\bx_K^{\mathsf{out}})-f(\bx^\star))\leq c_K(\|\bx^\star-\bx_0\|^2 + 2(\|\by^\star-\by_0\|^2+\|\by^\star\|^2)+\|\by_0\|^2)+d_K.
\end{equation}
Taking $y_i=0$ in \eqref{eq: mtw2} gives the upper bound for objective suboptimality as
    $f(\bx_K^{\mathsf{out}})-f(\bx^\star)\leq c_K(\|\bx^\star - \bx_0\|^2+\|\by_0\|^2) + d_K$
and combining this with \eqref{eq: axq33} finishes the proof.
  \end{proof}
  \begin{proof}[Proof of \Cref{cor: obj_feas}]
Because of the result of \Cref{lem: nlp1}, we note that the hypothesis of \Cref{lem: nlp2} is satisfied with 
\begin{subequations}
\begin{align}
    c_K &= \frac{3+\log(K+1)}{2\sum_{i=0}^K\tau_i},\label{eq: njp2} \\
    d_K &= \frac{1}{2\sum_{i=0}^{K-1}\tau_i}\Bigg[\sum_{k=0}^{K-1}\left( \frac{2\tau_k^2}{1-\beta_k}\|\widetilde{F}(\bz_k)\|^2 + 2\tau_k\langle F(\bz_k) - \widetilde{F}(\bz_k), \bz_k \rangle\right) +2\|\bz_0\|^2 \notag \\
    &\quad+ \bigg\|\sum_{k=0}^{K-1} \tau_k^2(F(\bz_k) - \widetilde{F}(\bz_k))\bigg\|^2 -\sum_{k=0}^{K-1}\frac{1-\beta_k}{2} \|\bz_{k+1} - \bz_k\|^2 + \beta_k\|\bz_0-\bz_{k+1}\|^2\Bigg].\label{eq: njp3}
\end{align}
    \end{subequations}

As a result, using the conclusion of \Cref{lem: nlp2}, after taking expectation of both sides, we have
\begin{subequations}\label{eq: sloeq}
\begin{align}
        \Eb|f(\xoutK) - f(\bx^\star)|\leq \Eb\left[c_K(\|\bx^\star-\bx_0\|^2 + 2(\|\by^\star-\by_0\|^2+\|\by^\star\|^2)+\|\by_0\|^2)+d_K\right],\label{eq: slo4}\\
        \sum_{i=1}^n \Eb[\max(0, g_i(\xoutK))] \leq \Eb\left[c_K (\|\bx^\star - \bx_0\|^2 + \| \by^\star+1-\by_0\|^2+\|\by_0\|^2) + d_K\right].\label{eq: slo5}
    \end{align}
\end{subequations}
Since $\sum_{i=0}^{K-1} \tau_i = O(\sqrt{K})$, the proof will be complete once we find a suitable bound for $\Eb[d_K]$. 

    First, by tower rule, we have
    \begin{align}
        \Eb \langle F(\bz_k) - \widetilde{F}(\bz_k), \bz_k \rangle  = \Eb  \langle F(\bz_k) - \Eb_k[\widetilde{F}(\bz_k)], \bz_k \rangle = 0.\label{eq: siw4}
    \end{align}
    Second, we estimate
    \begin{align}
        \Eb \bigg\|\sum_{k=0}^{K-1} \tau_k(F(\bz_k) - \widetilde{F}(\bz_k))\bigg\|^2 &= \sum_{k=0}^{K-1}\tau_k^2 \Eb\|F(\bz_k) - \widetilde{F}(\bz_k)\|^2 \leq \sum_{k=0}^{K-1}\tau_k^2 (B^2\|\bz_0-\bz_k\|^2 + G^2),\label{eq: siw5}
    \end{align}
    where the first identity is because $\Eb\langle F(\bz_k) - \widetilde{F}(\bz_k), F(\bz_j) - \widetilde{F}(\bz_j) \rangle=0$ for $j\neq k$ by tower rule. The inequality is by \Cref{asp: nlp}.

    We continue to estimate the terms in $\mathbb{E}[d_K]$. By \Cref{asp: nlp}, we also have
    \begin{align}
        \Eb\|\widetilde{F}(\bz_k)\|^2 &\leq B^2 \| \bz_0 - \bz_k\|^2 + G^2 \leq 2B^2 \| \bz_0 - \bz_{k+1}\|^2 + 2B^2 \| \bz_k - \bz_{k+1}\|^2 + G^2.\label{eq: siw6}
    \end{align}
    By the definitions of $\tau_k, \beta_k$, we have
    \begin{align}
        \frac{4\tau_k^2B^2}{1-\beta_k} + 2\tau_k^2 B^2 \leq \beta_k \text{~~~~and~~~~} \frac{4\tau_k^2B^2}{1-\beta_k}+2\tau_k^2B^2 \leq \frac{1-\beta_k}{2}.\label{eq: siw7}
    \end{align}
    Using \eqref{eq: siw4}, \eqref{eq: siw5}, \eqref{eq: siw6}, and \eqref{eq: siw7} in \eqref{eq: njp3} gives
        $\Eb[d_K] \leq \frac{1}{\sum_{i=0}^{K-1}\tau_i}\Big(\|\bz_0\|^2+\sum_{i=0}^{K-1} 5\tau_k^2G^2 \Big)$.
    Plugging in {$c_K$ from \eqref{eq: njp2}} and $\Eb[d_K]$ just derived to \eqref{eq: sloeq} and using the definition of $\tau_k$ give the result.
\end{proof}

\subsection{Residual Guarantees for Min-Max Problems} \label{sec: opnorm_main}

{In this section, we aim to provide guarantees for gradient norm-type optimality measures for min-max problems that may not be represented as the previous section.}
Consider the template given in \eqref{eq: minmax_temp} with $\mathcal{L}(\bx, \by) = \Psi(\bx, \by) + h_1(\bx) - h_2(\by)$ where $\Psi$ is smooth and $h_1, h_2$ are nonsmooth.  We can write this problem equivalently as follows:
\begin{equation}\label{eq: anh4}
    0\in F(\bz) + H(\bz), \text{~where~} \bz=\binom{\bx}{\by},~F(\bz)=\binom{\nabla_\bx \Psi(\bx, \by)}{-\nabla_\by \Psi(\bx, \by)} \text{~and~} H(\bz) = \binom{\partial h_1(\bx)}{\partial h_2(\by)}.
\end{equation}
 We next introduce an algorithm incorporating Halpern's idea to a version of the variance reduced method of \cite{pethick2023solving}. 
 In particular, we use extragradient \citep{korpelevich1976extragradient} (instead of Tseng's method \citep{tseng2000modified} used in \cite{pethick2023solving} who analyzed their method under the bounded variance assumption) and combine with the STORM variance reduction of \cite{cutkosky2019momentum}. {Let $h(\bz)=h_1(\bx) + h_2(\by)$.} 
 For  $k\geq 0$, we choose $\bz_0$ and define:
\begin{equation}\label{alg: halpern_eg_storm_app}
\begin{cases}
    \bar\bz_k &= \beta_k \bz_0 + (1-\beta_k) \bz_k \\
    \bz_{k+1/2} &= \prox_{\gamma_k h}(\bar\bz_k - \gamma_k \bg_k) \\
    \bz_{k+1} &= \prox_{\tau_k\gamma_k h}(\bar\bz_k - \tau_k \gamma_k \widetilde{F}(\bz_{k+1/2}, \xi_{k+1/2}))\\
    \bg_{k+1} &= \widetilde{F}(\bz_{k+1}, \xi_{k+1}) + (1-\alpha_k)(\bg_{k} - \widetilde{F}(\bz_{k}, \xi_{k+1})), \text{~where~} \xi_{t+1}\sim \Xi \text{~is i.i.d,}
\end{cases}
\end{equation}
where we recall the definition $\prox_{h}(\bx) = \arg\min_{\by} h(\by)+\frac{1}{2}\|\by-\bx\|^2$.

{In addition to the anchoring step for $\bar{\bz}_k$, we enhance the standard stochastic extragradient method by incorporating the variance reduced estimator of $\bg_k$. Note that the update of $\bz_{k+1}$ requires a smaller step size than the update of $\bz_{k+1/2}$, because (intuitively)  the estimate $\bg_k$ used in the update of $\bz_{k+1/2}$ is a more accurate estimate of the operator than the oracle $\widetilde{F}$ used in the update of $\bz_{k+1}$.}

{The main assumption of this section is as follows.}
\begin{assumption}\label{asp: minmax_opnorm}
    Let $\Psi\colon\mathbb{R}^{d+n}\to\mathbb{R}$ be convex-concave and smooth, and $h_1, h_2$ be convex, proper, and lower semicontinuous.
    We access an oracle $\widetilde{F}(\bz, \xi)$ such that
{\begin{enumerate}[itemsep=3pt]
        \item $\mathbb{E}[\widetilde{F}(\bz, \xi)] = F(\bz)$,
        \item $\mathbb{E}\|\widetilde{F}(\bz, \xi) - F(\bz)\|^2 \leq B^2 \| \bz-\bz_0\|^2 +G^2$,
        \item $\mathbb{E}\|\widetilde{F}(\bz, \xi) - \widetilde{F}(\bz', \xi)\|^2 \leq L^2 \| \bz-\bz'\|^2$.
    \end{enumerate} 
    }
\end{assumption}
{The third part of this assumption is stronger than what we required before, for example, by comparison with  \Cref{asp: 1}. This \emph{mean-square smoothness} assumption is widely used for variance reduced algorithms, see for example \cite{arjevani2023lower}. This assumption is also needed for analyses requiring bounded variance \cite{pethick2023solving} for the same problem as us.}

The following result describes the convergence of  \eqref{alg: halpern_eg_storm_app}.\footnote{Note that our result requires only  star-monotonicity of $F$, that is,
    $\langle F(\bz) - F(\bz^\star), \bz-\bz^\star \rangle  \geq 0$.
Even though this assumption is weaker than convex-concavity of $\Psi$,  for simplicity, we do not pursue this generalization.}
    The proof of the next theorem depends on \Cref{lem: one_it_before_vr} which appears later.
\begin{theorem}\label{th: minmax_opnorm}
    Let \Cref{asp: minmax_opnorm} hold and let $(\bz_{k+1/2})$ be generated by \eqref{alg: halpern_eg_storm_app} with \begin{equation*}
        \beta_k = \frac{\alpha_k}{2}= \frac{1}{k+3},~~~ \gamma_k = \frac{1-\beta_k}{6L},~~~ \tau_k = \frac{\tau}{\sqrt{k+3}} := \frac{1}{\sqrt{k+3}}\min\left(\frac{1}{4}, \frac{L^2}{12B^2}\right).
    \end{equation*}
    Then, {with $K\geq 2$,} we have in view of \eqref{eq: anh4} that
    \begin{align*}
        &\frac{1}{\sum_{i=0}^{K-1}\sqrt{i+3}}\sum_{k=0}^{K-1}\sqrt{k+3} \, \mathbb{E}[\dist^2(0, (F+H)\bz_{k+1/2})] \\
        &\quad\leq {\frac{1221}{((K+2)^{3/2} - 3^{3/2})\tau} \left((K+2)L^2\|\bz_0-\bz^\star\|^2 + \tau \| \bg_0 - F(\bz_0)\|^2 + K(\tau+2)^2G^2\right)} \\
        &\quad= O\left(\frac{1}{\sqrt{K}} \right).
    \end{align*}
 Equivalently, we  have $\mathbb{E}[\res_{\hat k}^2]=\mathbb{E}[\dist^2(0, (F+H)\bz_{\hat k+1/2})] \leq \varepsilon^2$ with stochastic first-order oracle complexity $O(\varepsilon^{-4})$ where $\hat k \in [0, \dots, K-1]$ is selected as $\Pr(\hat k = k) = \frac{\sqrt{k+3}}{\sum_{i=0}^{K-1}\sqrt{i+3}}$. 
\end{theorem}
\begin{remark}
    {For \emph{unconstrained} convex-concave problems, a complexity of $\widetilde{O}(\varepsilon^{-2})$ was obtained in \cite{chen2022near} for gradient norm under the bounded variance assumption. However, for \emph{constrained} convex-concave problems, even with the bounded variance assumption, the complexity $O(\varepsilon^{-4})$ for residual norm is the best-known. It is often obtained with multi-loop algorithms or increasing mini-batch sizes \citep{lee2021fast,pethick2023solving}. Hence, our result in this theorem obtains this best-known complexity for constrained problems, by relaxing the bounded variance assumption.}
    The only result that we are aware with a residual-type guarantee without bounded variance is \cite[Theorem 4.5]{choudhury2024single}, where the authors focus on a restrictive \emph{unconstrained} min-max problem and require mini-batch sizes to be set depending on $\max\left( K, \frac{K\Eb\|\widetilde{F}(\bz^\star, \xi)\|^2}{\|\bz_0-\bz^\star\|^2} \right)$, where $K$ is the final iteration counter and $\bz^\star$ is a solution.
    Our results can also be extended in a straightforward way to solve variational inequalities.
\end{remark}
\begin{proof}[Proof of \Cref{th: minmax_opnorm}]
We start from the result of \Cref{lem: one_it_before_vr}. First, we estimate the terms in $\bad_k$. We have by \Cref{asp: minmax_opnorm} that
\begin{align*}
    &\frac{\tau_k^2(1-\beta_k)}{6L^2} \mathbb{E} \| F(\bz_{k+1/2}) - \widetilde{F}(\bz_{k+1/2}, \xi_{k+1/2})\|^2 \leq \frac{\tau_k^2(1-\beta_k)}{6L^2}\left(B^2\Eb\|\bz_{k+1/2}-\bz_0\|^2 +G^2\right).
\end{align*}
Then, using Young's inequality for the first term inside the paranthesis and using $\frac{2\tau_k^2(1-\beta_k)B^2}{6L^2} \leq \tau_k$ and $\frac{2\tau_k^2(1-\beta_k)B^2}{6L^2} \leq \frac{\beta_k}{4}$, which are due to
$\tau_k^2 = \tau^2 \beta_k$, $\tau=\min\left\{\frac{1}{4}, \frac{L^2}{12B^2}\right\}$, $(1-\beta_k)\leq 1$, we get
\begin{equation*}
    \frac{\tau_k^2(1-\beta_k)}{6L^2} \mathbb{E} \| F(\bz_{k+1/2}) - \widetilde{F}(\bz_{k+1/2}, \xi_{k+1/2})\|^2 
     \leq \tau_k\Eb\| \bz_{k+1/2}-\bz_{k+1}\|^2 + \frac{\beta_k}{4} \Eb\|\bz_{k+1}-\bz_0\|^2+\frac{\tau_k^2 G^2}{6L^2}.
\end{equation*}
Second, substituting $\alpha_k = \frac{2}{\sqrt{k+3}}$ in \Cref{lem: storm} and multiplying the result by $\frac{\tau(1-\beta_k)}{4L^2}$ give
\begin{align*}
    \frac{\tau(1-\beta_k)}{4L^2\sqrt{k+3}} \Eb\|\bg_k-F(\bz_k)\|^2 &\leq \frac{\tau (1-\beta_k)}{4L^2}\Big( 1-\frac{1}{\sqrt{k+3}} \Big) \Eb\|\bg_k - F(\bz_k)\|^2 \\
    &\quad- \frac{\tau (1-\beta_k)}{4L^2}\Eb\|\bg_{k+1} - F(\bz_{k+1})\|^2 + \frac{\tau(1-\beta_k)}{2} \Eb\|\bz_{k+1} - \bz_k\|^2 \\
    &\quad  + \frac{2\tau(1-\beta_k)}{L^2(k+3)}\left( B^2\Eb\|\bz_k-\bz_0\|^2 + G^2 \right).
\end{align*}
Combining the last two estimates, in view of the definition of $\bad_k$ in \Cref{lem: one_it_before_vr}, give
\begin{align}
    \bad_k &\leq \frac{\tau (1-\beta_k)}{4L^2}\Big( 1-\frac{1}{\sqrt{k+3}} \Big) \Eb\|\bg_k - F(\bz_k)\|^2  \notag \\
    &\quad- \frac{\tau (1-\beta_k)}{4L^2}\Eb\|\bg_{k+1} - F(\bz_{k+1})\|^2- \frac{\tau_k(1-\beta_k)}{12L^2} \mathbb{E}\|\bg_k-F(\bz_k)\|^2 \notag \\
    &\quad + \frac{\tau(1-\beta_k)}{2} \Eb\|\bz_{k+1} - \bz_k\|^2 + \frac{2\tau(1-\beta_k)}{L^2(k+3)}\left( B^2\Eb\|\bz_k-\bz_0\|^2 + G^2 \right)  \notag \\
    &\quad+ \tau_k\Eb\| \bz_{k+1/2}-\bz_{k+1}\|^2 + \frac{\beta_k}{4} \Eb\|\bz_{k+1}-\bz_0\|^2 +  \frac{\tau_k^2G^2}{6L^2},\label{eq: sqo4}
\end{align}
where we also used $\tau_k  =\frac{\tau}{\sqrt{k+3}}$.
Next, Young's inequality, $\beta_k = \frac{1}{k+3} \leq \frac{1}{3}$, and $\tau<\frac{L^2}{12B^2}$ give
\begin{align*}
    \frac{2B^2\tau(1-\beta_k)}{L^2(k+3)}\Eb \| \bz_k - \bz_0\|^2 &\leq \frac{2B^2\tau(1-\beta_k)}{L^2(k+3)}\left(3\Eb \| \bz_{k+1} - \bz_0\|^2+\frac{3}{2}\Eb \| \bz_k - \bz_{k+1}\|^2 \right) \\
    &\leq \frac{\beta_k}{2}\Eb \| \bz_{k+1} - \bz_0\|^2+\frac{1-\beta_k}{8}\Eb \| \bz_k - \bz_{k+1}\|^2.
\end{align*}
By substituting the last estimate into \eqref{eq: sqo4} and combining with the definition of $\good_k$ from \Cref{lem: one_it_before_vr}, we have
\begin{align}
    &\bad_k - \good_k\notag \\
    &\leq \frac{\tau (k+1)}{4L^2(k+3)}\Eb\|\bg_k - F(\bz_k)\|^2 - \frac{\tau (k+2)}{4L^2(k+3)}\Eb\|\bg_{k+1} - F(\bz_{k+1})\|^2 + \frac{(\tau+2)^2G^2}{(k+3)L^2} \notag \\
    &\quad -\tau_k\underbrace{\Big(\frac{1-\beta_k}{12L^2} \mathbb{E}\|\bg_k-F(\bz_k)\|^2 +\frac{1-\beta_k}{3}\Eb\|\bz_k-\bz_{k+1/2}\|^2+ \beta_k\Eb\|\bz_0-\bz_{k+1/2}\|^2\Big)}_{\mathcal{R}_k} ,\label{eq: ygh4}
\end{align}
where the terms involving $\|\bz_k-\bz_{k+1}\|^2$ disappeared because of $\tau < 1/4$ and we used
\begin{equation*}
    (1-\beta_k)\Big(1-\frac{1}{\sqrt{k+3}}\Big) = \frac{k+2}{k+3}\Big(1-\frac{1}{\sqrt{k+3}}\Big) \leq \frac{k+1}{k+3},
\end{equation*}
for $k\geq 0$, which follows from the definition of $\beta_k$.
Let us denote
\begin{equation*}
    \Phi_k = (k+2)\mathbb{E}\|\bz^\star - \bz_{k}\|^2 + \frac{\tau(k+1)}{4L^2} \mathbb{E}\| \bg_k - F(\bz_k)\|^2.
\end{equation*}
With this, by substituting \eqref{eq: ygh4} into the result of \Cref{lem: one_it_before_vr} and using $\tau\leq1/4$, we obtain
\begin{align*}
    &\frac{1}{k+3}\Phi_{k+1}\leq \frac{1}{k+3}\Phi_k + \beta_k\|\bz^\star-\bz_0\|^2 +\frac{(\tau+2)^2G^2}{(k+3)L^2}-\tau_k\mathcal{R}_k.
\end{align*}
We then multiply both sides by $k+3$ to obtain
    $\Phi_{k+1} \leq \Phi_k + \|\bz^\star - \bz_0\|^2  +\frac{(\tau+2)^2G^2}{L^2}-(k+3)\tau_k\resid_k$.
By summing this inequality for $k=0,\dots,K-1$, we obtain
\begin{align*}
    \sum_{k=0}^{K-1}(k+3)\tau_k\resid_k
    \leq (K+2)\Eb\|\bz^\star - \bz_0\|^2 + \frac{\tau}{4L^2}\|\bg_0 - F(\bz_0)\|^2 + \frac{K(\tau+2)^2G^2}{L^2}.
\end{align*}
Using \Cref{lem: optimality_measure} with $\theta_k=(k+3)\tau_k=\tau\sqrt{k+3}=\Theta(\sqrt{k+3})$ and $\sum_{k=0}^{K-1}\sqrt{k+3} = \Omega (K^{3/2})$ gives the result.
\end{proof}

Our technical result for analyzing one iteration of the algorithm is as follows.
\begin{lemma}\label{lem: one_it_before_vr}
    Let \Cref{asp: minmax_opnorm} hold. Let $(\bz_k, \bz_{k+1/2})$ be generated by \eqref{alg: halpern_eg_storm_app} with parameters given in \Cref{th: minmax_opnorm}. Then, we have
\begin{align*}
    \Eb\|\bz^\star &- \bz_{k+1}\|^2 \leq (1-\beta_k) \Eb\|\bz^\star - \bz_k\|^2 + \beta_k \| \bz^\star - \bz_0\|^2 - \good_k + \bad_k,
\end{align*}
where
\begin{align*}
    \good_k &= \beta_k\tau_k \|\bz_0-\bz_{k+1/2}\|^2 + \frac{1-\beta_k}{4}\Eb\|\bz_{k} - \bz_{k+1}\|^2 + \frac{3\beta_k}{4}\Eb\|\bz_{k+1} - \bz_0\|^2\\
    &\quad+\frac{\tau_k(1-\beta_k)}{3}\Eb\|\bz_k-\bz_{k+1/2}\|^2 + \tau_k \Eb\|\bz_{k+1} - \bz_{k+1/2}\|^2,\\
    \bad_k &= \frac{\tau_k(1-\beta_k)}{6L^2}\Eb\big[\|\bg_k-F(\bz_k)\|^2+\tau_k\|F(\bz_{k+1/2}) - \widetilde{F}(\bz_{k+1/2}, \xi_{k+1/2})\|^2\big].
\end{align*}
\end{lemma}
\begin{proof}
Definitions of $\bz_{k+1/2}$ and $\bz_{k+1}$, along with the definition of the proximal operator give us that
\begin{align}
    \langle \bz_{k+1} - \bar\bz_k + \tau_k \gamma_k \widetilde{F}(\bz_{k+1/2}, \xi_{k+1/2}), \bz^\star - \bz_{k+1} \rangle &\geq \tau_k\gamma_k(h(\bz_{k+1})-h(\bz^\star)),\label{eq: asl4} \\
    \langle \bz_{k+1/2} - \bar\bz_k + \gamma_k \bg_k, \bz_{k+1} -\bz_{k+1/2} \rangle &\geq \gamma_k (h(\bz_{k+1/2})-h(\bz_{k+1})).\label{eq: asl5}
\end{align}
We multiply \eqref{eq: asl4} with $2$ and \eqref{eq: asl5} with $2\tau_k$ and combine the inequalities to obtain
\begin{align}
    &0\leq 2\langle \bz_{k+1} - \bar\bz_k, \bz^\star- \bz_{k+1} \rangle + 2\tau_k \langle \bz_{k+1/2} - \bar\bz_k, \bz_{k+1} - \bz_{k+1/2} \rangle\notag \\
    &+2\tau_k\gamma_k\big(\langle \widetilde{F}(\bz_{k+1/2}, \xi_{k+1/2}), \bz^\star - \bz_{k+1} \rangle + \langle \bg_k, \bz_{k+1} - \bz_{k+1/2} \rangle +h(\bz^\star) - h(\bz_{k+1/2})\big) .\label{eq: asl6}
\end{align}
For the first inner product in \eqref{eq: asl6}, we use \Cref{fact: beta} with $\ba \leftarrow \bz_{k+1}$, $\bar\msfx_k \leftarrow \bar\bz_k$, and $\bb \leftarrow \bz^\star$:
\begin{align}
    2\langle \bz_{k+1} - \bar \bz_k, \bz^\star-\bz_{k+1} \rangle 
    &= -\| \bz^\star-\bz_{k+1}\|^2 - \beta_k \| \bz_{k+1} - \bz_0\|^2 + \beta_k \| \bz^\star-\bz_0\|^2\notag \\
    &\quad- (1-\beta_k) \| \bz_{k+1} - \bz_k\|^2 + (1-\beta_k) \| \bz^\star-\bz_k\|^2.\label{eq: asl7}
\end{align}
By using \Cref{fact: beta} with $\ba \leftarrow \bz_{k+1/2}$, $\bar\msfx_k \leftarrow \bar\bz_k$, and $\bb\leftarrow\bz_{k+1}$, we similarly have the following for the second inner product in \eqref{eq: asl6}:
\begin{align}
    &2\tau_k\langle \bz_{k+1/2} - \bar \bz_k, \bz_{k+1} - \bz_{k+1/2} \rangle \notag \\
    &= -\tau_k\| \bz_{k+1} - \bz_{k+1/2}\|^2 - \beta_k \tau_k\| \bz_{k+1/2} - \bz_0\|^2 + \beta_k\tau_k \| \bz_{k+1} - \bz_0\|^2 \notag \\
    &\quad - \tau_k(1-\beta_k) \| \bz_{k+1/2} - \bz_k\|^2 + \tau_k(1-\beta_k) \| \bz_{k+1} - \bz_k\|^2. \label{eq: asl8}
\end{align}
For the remaining terms in \eqref{eq: asl6}, let us note the following
\begin{align}
    &2\tau_k\gamma_k \Eb \big(\langle \widetilde{F}(\bz_{k+1/2}, \xi_{k+1/2}), \bz^\star - \bz_{k+1} \rangle + h(\bz^\star)-h(\bz_{k+1/2}) + \langle \bg_k, \bz_{k+1} - \bz_{k+1/2 }\rangle \big) \notag\\
    &\leq 2\tau_k\gamma_k\Eb\langle \bg_k - \widetilde{F}(\bz_{k+1/2}, \xi_{k+1/2}), \bz_{k+1} - \bz_{k+1/2} \rangle \notag \\
    &= 2\tau_k\gamma_k\Eb\langle \bg_k - F(\bx_{k+1/2}), \bz_{k} - \bz_{k+1/2} \rangle + 2\tau_k\gamma_k\Eb\langle \bg_k - \widetilde{F}(\bz_{k+1/2}, \xi_{k+1/2}), \bz_{k+1} - \bz_{k} \rangle\notag.
    \end{align}
The inequality used the tower rule to get $\mathbb{E}\langle \widetilde{F}(\bz_{k+1/2}, \xi_{k+1/2}), \bz^\star-\bz_{k+1/2} \rangle=\mathbb{E}\langle F(\bz_{k+1/2}), \bz^\star-\bz_{k+1/2} \rangle$, monotonicity of $F$, and the definition of $\bz^\star$ as the solution to the variational inequality
$\langle F(\bz^\star), \bz^\star - \bz \rangle + h(\bz^\star) - h(\bz) \leq 0~\forall \bz$. The equality is by adding and subtracting $\bz_k$ and using tower rule and the fact that $\bz_k-\bz_{k+1/2}$ is deterministic under the conditioning. 

Next, we further bound the right-hand side by Young's inequality to obtain
    \begin{align}
 &2\tau_k\gamma_k \Eb \big(\langle \widetilde{F}(\bz_{k+1/2}, \xi_{k+1/2}), \bz^\star - \bz_{k+1} \rangle + h(\bz^\star)-h(\bz_{k+1/2}) + \langle \bg_k, \bz_{k+1} - \bz_{k+1/2 }\rangle \big) \notag \\
    &\leq \frac{\tau_k(1-\beta_k)}{2} \Eb \| \bz_k - \bz_{k+1/2}\|^2 + \frac{2\tau_k\gamma_k^2}{1-\beta_k}\Eb\|\bg_k - F(\bz_{k+1/2})\|^2 \notag \\
    &\quad + \frac{1-\beta_k}{2} \Eb \| \bz_{k+1} - \bz_k\|^2 + \frac{2\tau_k^2\gamma_k^2}{1-\beta_k} \Eb \| \bg_k - \widetilde{F}(\bz_{k+1/2}, \xi_{k+1/2})\|^2.\label{eq: asl9}
\end{align}
To estimate \eqref{eq: asl9}, we use  Young's inequality and Lipschitzness of $F$ which give us that
    $\Eb\|\bg_k - F(\bz_{k+1/2})\|^2 \leq 2\Eb[\|\bg_k - F(\bz_{k})\|^2 + L^2\|\bz_k - \bz_{k+1/2}\|^2].$
By $\gamma_k^2 = \frac{(1-\beta_k)^2}{36L^2}$, this implies
\begin{align}
    \frac{2\tau_k\gamma_k^2}{1-\beta_k}\Eb\|\bg_k - F(\bz_{k+1/2})\|^2 \leq \frac{\tau_k(1-\beta_k)}{9} \Eb \Big[ \frac{1}{L^2}\|\bg_k - F(\bz_{k})\|^2 + \|\bz_k - \bz_{k+1/2}\|^2\Big].\label{eq: ogh5}
\end{align}
To further estimate \eqref{eq: asl9}, Lipschitzness of $F$ and Young's inequality also gives
\begin{align*}
     \Eb \| \bg_k - \widetilde{F}(\bz_{k+1/2}, \xi_{k+1/2})\|^2 &\leq 3 \Eb \| \bg_k - {F}(\bz_{k})\|^2 + 3L^2 \Eb \| \bz_k - \bz_{k+1/2}\|^2 \\
     &\quad+ 3\Eb \| F(\bz_{k+1/2}) - \widetilde{F}(\bz_{k+1/2}, \xi_{k+1/2})\|^2.
\end{align*}
By $\tau_k\leq \frac{1}{4}$ and $\gamma_k^2 = \frac{(1-\beta_k)^2}{36L^2}$, this implies
\begin{align}
     \frac{2\tau_k^2\gamma_k^2}{1-\beta_k}\Eb \| \bg_k - \widetilde{F}(\bz_{k+1/2}, \xi_{k+1/2})\|^2 &\leq \frac{\tau_k(1-\beta_k)}{24}\Eb\Big[ \frac{1}{L^2} \| \bg_k - {F}(\bz_{k})\|^2 +  \| \bz_k - \bz_{k+1/2}\|^2 \Big]\notag \\
     &\quad+ \frac{\tau_k^2(1-\beta_k)}{6L^2}\Eb \| F(\bz_{k+1/2}) - \widetilde{F}(\bz_{k+1/2}, \xi_{k+1/2})\|^2.\label{eq: ogh6}
\end{align}
By combining \eqref{eq: asl7}, \eqref{eq: asl8}, \eqref{eq: asl9} in \eqref{eq: asl6}, using \eqref{eq: ogh5} and \eqref{eq: ogh6} to upper bound the right-hand side of \eqref{eq: asl9}, and using $\tau_k \leq 1/4$, we obtain the result.
\end{proof}
\section{Conclusions and open questions}\label{sec: conclusions}
We have shown that an insight from \cite{neu2024dealing} which provided a connection between Halpern iteration and classical \eqref{eq: bg} assumption helps improving our understanding of the behavior of stochastic algorithms for minimization and min-max optimization without bounded variance or bounded domain assumptions.
One can also use our ideas in \Cref{sec: main_min} along with \Cref{sec: opnorm_main} to show similar guarantees for additively composite template  $\min_\bx \, f(\bx)+g(\bx)$,
where $g$ is a nonsmooth function with an efficient proximal operator and  $f$ is convex and smooth. 

While the techniques used in this paper integrate well with convexity, it is not clear how to remove the additional error coming from \eqref{eq: bg} in the nonconvex cases with a direct analysis, one that avoids the regularization device used in \cite{allen2018make}, for example. 
The reason is that the proof templates for nonconvex minimization utilize the descent lemma (see \cite[Lemma 1.2.3]{nesterov2018lectures}) and they do not contain the \emph{good} terms of the form used in \Cref{lem: one_iteration_subg}.

In the setting of \Cref{sec: main_minmax}, there exist results for \emph{unconstrained} min-max optimization with a bounded-variance assumption where one can improve the $O(\varepsilon^{-4})$ complexity; see \cite{chen2022near}. 
It is an open question to derive improved guarantees for the residual in the more interesting \emph{constrained} case, even with bounded-variance assumptions. 
Once this is achieved, the ideas in this paper can then be used to obtain improved complexities for constrained min-max optimization without bounded-variance assumption. 

For minimization, it seems to be open to get a tight rate for SGD under \eqref{eq: bg0} without  a fixed horizon. The analysis of \cite{domke2023provable} under a similar assumption relied on using a fixed horizon.

\section{Additional Proofs}\label{sec: opnorm_proof}
The following lemma is extracted  from the analysis in \cite{cutkosky2019momentum}.
\begin{lemma}\label{lem: storm}
    Let \Cref{asp: minmax_opnorm} hold.
    For $(\bg_{k+1})$ defined in \eqref{alg: halpern_eg_storm_app}, we have for $k\geq 0$ that
    \begin{align*}
        \frac{\alpha_k}{2}\Eb\|\bg_k - F(\bz_k)\|^2 &\leq \left(1-\frac{\alpha_k}{2}\right)\Eb\|\bg_k - F(\bz_k)\|^2 - \Eb\|\bg_{k+1}-F(\bz_{k+1})\|^2 \\
        &\quad+ 2L^2\Eb\|\bz_{k+1} - \bz_k\|^2 + 2\alpha_k^2 (B^2 \Eb\|\bz_k-\bz_0\|^2+ G^2).
    \end{align*}
\end{lemma}
\begin{proof}
On the definition of $\bg_{k+1}$ in \eqref{alg: halpern_eg_storm_app}, we subtract $F(\bz_{k+1})$ from both sides to obtain
    \begin{align*}
        \bg_{k+1} - F(\bz_{k+1}) = \tilde F(\bz_{k+1},\xi_{k+1}) - F(\bz_{k+1}) + (1-\alpha_{k})(\bg_{k} - F(\bz_{k}) + F(\bz_{k}) - \widetilde{F}(\bz_{k}, \xi_{k+1})).
    \end{align*}
    By taking the squared norm and expectation, we obtain
    \begin{align}
        &\Eb\|\bg_{k+1} - F(\bz_{k+1})\|^2 = (1-\alpha_{k})^2 \Eb\| \bg_{k} - F(\bz_{k})\|^2 \notag\\
        &+ 2(1-\alpha_{k}) \Eb\langle \bg_{k} - F(\bz_{k}), \widetilde{F}(\bz_{k+1}, \xi_{k+1}) - F(\bz_{k+1}) + (1-\alpha_{k})(F(\bz_{k}) - F(\bz_{k}, \xi_{k+1})) \rangle \notag\\
        &+ \Eb\| \widetilde{F}(\bz_{k+1}, \xi_{k+1}) - F(\bz_{k+1}) + (1-\alpha_{k})(F(\bz_{k}) - \widetilde{F}(\bz_{k}, \xi_{k+1}))\|^2.\label{eq: som4}
    \end{align}
For the final term on the right-hand side, Young's inequality and Jensen's inequality give
    \begin{align}
        &\Eb\| \widetilde{F}(\bz_{k+1}, \xi_{k+1}) - F(\bz_{k+1}) + (1-\alpha_{k})(F(\bz_{k}) - \widetilde{F}(\bz_{k}, \xi_{k+1}))\|^2\notag \\
        &\leq 2\Eb[ \|\widetilde{F}(\bz_{k+1}, \xi_{k+1}) - F(\bz_{k+1}) - \widetilde{F}(\bz_k, \xi_{k+1}) + F(\bz_k) \|^2 +\alpha_k^2 \|F(\bz_k) - \widetilde{F}(\bz_k, \xi_{k+1})\|^2]\notag \\
        &\leq 2\Eb \|\Ftilde(\bz_{k+1}, \xi_{k+1}) - \Ftilde(\bz_k, \xi_{k+1})\|^2 + 2\alpha_k^2 \Eb\|F(\bz_k) - \widetilde{F}(\bz_k, \xi_{k+1})\|^2\notag  \\
        &\leq 2L^2\Eb \|\bz_{k+1}-\bz_k\|^2 + 2\alpha_k^2 (B^2\Eb\|\bz_k-\bz_0\|^2 + G^2).\label{eq: som6}
    \end{align}
    Moreover, after using tower rule, we will have that the inner product on the right-hand side of \eqref{eq: som4} will be zero in expectation because $\Eb_k[\widetilde{F}(\bz_{k+1}, \xi_{k+1})] = F(\bz_{k+1})$ and $\Eb_{k}[\widetilde{F}(\bz_{k}, \xi_{k+1})] = F(\bz_{k})$. 
    By using this argument and \eqref{eq: som6} in \eqref{eq: som4}, together with $\alpha_k \le 1$, the result follows.
\end{proof}
\begin{lemma}\label{lem: optimality_measure}
    Given $\bz_{k+1/2}$ from \eqref{alg: halpern_eg_storm_app}, $\beta_k =\frac{1}{k+3}$, and $\mathcal{R}_k$ from \eqref{eq: ygh4}, {that is,
    \begin{equation*}
    \mathcal{R}_k= \frac{1-\beta_k}{12L^2} \mathbb{E}\|\bg_k-F(\bz_k)\|^2 +\frac{1-\beta_k}{3}\Eb\|\bz_k-\bz_{k+1/2}\|^2+ \beta_k\Eb\|\bz_0-\bz_{k+1/2}\|^2.\end{equation*}}
    For $\theta_k >0$, we get
    \begin{equation}\label{eq: sagf4}
        \theta_k \mathsf{c}^{-1} \dist^2(0, (F+H)\bz_{k+1/2})\leq \theta_k \mathsf{c}^{-1}\|F(\bz_{k+1/2}) + h_{k+1/2}\|^2 \leq \theta_k \mathcal{R}_k,
    \end{equation}
    where $h_{k+1/2} := \gamma_k^{-1} (\bar\bz_k - \bz_{k+1/2}) - \bg_k \in H(\bz_{k+1/2})$
    and {$\mathsf{c} := 747 L^2$}.
\end{lemma}
\begin{proof}
    We have by Young's inequality that
    \begin{align*}
        \|F(\bz_{k+1/2}) + h_{k+1/2}\|^2 &\leq  4 \| \bg_k-F(\bz_k)\|^2 +2\gamma_k^{-2}\|\bar\bz_k - \bz_{k+1/2}\|^2 +  4 \| F(\bz_k) - F(\bz_{k+1/2})\|^2\\
        &\leq 4 \| \bg_k-F(\bz_k)\|^2 + 166L^2\|\bz_k - \bz_{k+1/2}\|^2  + 162\beta_k L^2 \| \bz_0 - \bz_{k+1/2}\|^2
        \\
        &{\leq  \frac{\mathsf{c}}{12 L^2} \cdot \frac{2}{3} \| \bg_k-F(\bz_k)\|^2 + \frac{\mathsf{c}}{3} \cdot \frac23 \|\bz_k - \bz_{k+1/2}\|^2  + \mathsf{c} \beta_k \| \bz_0 - \bz_{k+1/2}\|^2 } \\
        &\leq \mathsf{c}\mathcal{R}_k,
    \end{align*}
    where we used $\|\bar\bz_k - \bz_{k+1/2}\|^2 \leq \beta_k \| \bz_0 - \bz_{k+1/2} \|^2 + (1-\beta_k) \|\bz_k - \bz_{k+1/2}\|^2$, Lipschitzness of $F$, with $\gamma_k^2=\frac{(1-\beta_k)^2}{36L^2}$, the definition of $\mathsf{c}$, and $\frac23\leq1-\beta_k\leq 1$. 
    Multiplying both sides by $\theta_k\mathsf{c}^{-1}$ gives the second inequality in \eqref{eq: sagf4}.
    The definition of $\bz_{k+1/2}$ gives
        $\bz_{k+1/2} + \gamma_k H(\bz_{k+1/2}) \ni \bar\bz_k - \gamma_k \bg_k \iff h_{k+1/2}\in H(\bz_{k+1/2})$.
    This completes the proof.
\end{proof}
\begin{fact}\label{fact: beta}
    Let $\ba, \bb$ be arbitrary and let us set $\bar\msfx_k = (1-\beta_k)\msfx_k + \beta_k \msfx_0$. Then, we have
    \begin{align*}
        2\langle \ba-\bar\msfx_k, \bb-\ba \rangle &= - \| \bb-\ba\|^2 + (1-\beta_k) \| \bb-\msfx_k\|^2 + \beta_k \| \bb-\msfx_0\|^2 \\
        &\quad-\beta_k\|\ba-\msfx_0\|^2 - (1-\beta_k) \| \ba-\msfx_k\|^2.
    \end{align*}
\end{fact}
\begin{proof}
    By the definition of $\bar\msfx_k$, we have
    \begin{align*}
        2\langle \ba-\bar\msfx_k, \bb-\ba \rangle = 2\beta_k \langle \ba-\msfx_0, \bb-\ba \rangle + 2(1-\beta_k) \langle \ba-\msfx_k, \bb-\ba \rangle.
    \end{align*}
    Using $2\langle \bz, \by\rangle = \|\bz + \by \|^2 - \| \bz \|^2  - \| \by\|^2$ twice on the right-hand side completes the proof.
\end{proof}
{
\section*{Acknowledgments}
We thank Francesco Orabona and Jelena Diakonikolas for suggesting references \cite{polyak1973pseudogradient} and \cite{nemirovski1983problem}, respectively, for some of the results we presented from the literature. We thank two referees whose comments improved the manuscript significantly.

This research was funded by the Natural Sciences and Engineering Research Council of Canada (NSERC), [funding reference number RGPIN-2025-06634]; by the NSF Awards DMS 2023239 and CCF 2224213; and by the Austrian Science Fund (FWF) 10.55776/STA223.

Part of the work was done while A. Alacaoglu was at the University of Wisconsin--Madison.

}

 \bibliographystyle{plainnat}

\bibliography{lit}

\begin{thebibliography}{66}
\providecommand{\natexlab}[1]{#1}
\providecommand{\url}[1]{\texttt{#1}}
\expandafter\ifx\csname urlstyle\endcsname\relax
  \providecommand{\doi}[1]{doi: #1}\else
  \providecommand{\doi}{doi: \begingroup \urlstyle{rm}\Url}\fi

\bibitem[Allen-Zhu(2018)]{allen2018make}
Zeyuan Allen-Zhu.
\newblock How to make the gradients small stochastically: Even faster convex
  and nonconvex sgd.
\newblock \emph{Advances in Neural Information Processing Systems}, 31, 2018.

\bibitem[Allen-Zhu and Yuan(2016)]{allen2016improved}
Zeyuan Allen-Zhu and Yang Yuan.
\newblock Improved svrg for non-strongly-convex or sum-of-non-convex
  objectives.
\newblock In \emph{International conference on machine learning}, pages
  1080--1089. PMLR, 2016.

\bibitem[Arjevani et~al.(2023)Arjevani, Carmon, Duchi, Foster, Srebro, and
  Woodworth]{arjevani2023lower}
Yossi Arjevani, Yair Carmon, John~C Duchi, Dylan~J Foster, Nathan Srebro, and
  Blake Woodworth.
\newblock Lower bounds for non-convex stochastic optimization.
\newblock \emph{Mathematical Programming}, 199\penalty0 (1-2):\penalty0
  165--214, 2023.

\bibitem[Asi and Duchi(2019)]{asi2019stochastic}
Hilal Asi and John~C Duchi.
\newblock Stochastic (approximate) proximal point methods: Convergence,
  optimality, and adaptivity.
\newblock \emph{SIAM Journal on Optimization}, 29\penalty0 (3):\penalty0
  2257--2290, 2019.

\bibitem[Bach and Moulines(2011)]{moulines2011non}
Francis Bach and Eric Moulines.
\newblock Non-asymptotic analysis of stochastic approximation algorithms for
  machine learning.
\newblock \emph{Advances in neural information processing systems}, 24, 2011.

\bibitem[Bauschke and Combettes(2017)]{bauschke2017convex}
Heinz~H Bauschke and Patrick~L Combettes.
\newblock Convex analysis and monotone operator theory in hilbert spaces.
\newblock \emph{CMS Books in Mathematics}, 2017.

\bibitem[Bertsekas and Tsitsiklis(2000)]{bertsekas2000gradient}
Dimitri~P Bertsekas and John~N Tsitsiklis.
\newblock Gradient convergence in gradient methods with errors.
\newblock \emph{SIAM Journal on Optimization}, 10\penalty0 (3):\penalty0
  627--642, 2000.

\bibitem[Blum(1954)]{blum1954approximation}
Julius~R Blum.
\newblock Approximation methods which converge with probability one.
\newblock \emph{The Annals of Mathematical Statistics}, pages 382--386, 1954.

\bibitem[Bottou et~al.(2018)Bottou, Curtis, and
  Nocedal]{bottou2018optimization}
L{\'e}on Bottou, Frank~E Curtis, and Jorge Nocedal.
\newblock Optimization methods for large-scale machine learning.
\newblock \emph{SIAM review}, 60\penalty0 (2):\penalty0 223--311, 2018.

\bibitem[Cai et~al.(2022)Cai, Song, Guzm{\'a}n, and
  Diakonikolas]{cai2022stochastic}
Xufeng Cai, Chaobing Song, Crist{\'o}bal Guzm{\'a}n, and Jelena Diakonikolas.
\newblock Stochastic halpern iteration with variance reduction for stochastic
  monotone inclusions.
\newblock \emph{Advances in Neural Information Processing Systems},
  35:\penalty0 24766--24779, 2022.

\bibitem[Cai and Zheng(2023)]{cai2023accelerated}
Yang Cai and Weiqiang Zheng.
\newblock Accelerated single-call methods for constrained min-max optimization.
\newblock In \emph{The Eleventh International Conference on Learning
  Representations, 2023}, 2023.

\bibitem[Cai et~al.(2024)Cai, Oikonomou, and Zheng]{cai2022accelerateda}
Yang Cai, Argyris Oikonomou, and Weiqiang Zheng.
\newblock Accelerated algorithms for constrained nonconvex-nonconcave min-max
  optimization and comonotone inclusion.
\newblock In \emph{International Conference on Machine Learning}, pages
  5312--5347. PMLR, 2024.

\bibitem[Chen and Luo(2024)]{chen2022near}
Lesi Chen and Luo Luo.
\newblock Near-optimal algorithms for making the gradient small in stochastic
  minimax optimization.
\newblock \emph{Journal of Machine Learning Research}, 25\penalty0
  (387):\penalty0 1--44, 2024.

\bibitem[Choudhury et~al.(2024)Choudhury, Gorbunov, and
  Loizou]{choudhury2024single}
Sayantan Choudhury, Eduard Gorbunov, and Nicolas Loizou.
\newblock Single-call stochastic extragradient methods for structured
  non-monotone variational inequalities: Improved analysis under weaker
  conditions.
\newblock \emph{Advances in Neural Information Processing Systems}, 36, 2024.

\bibitem[Cohen and Zhu(1983)]{cohen1983decomposition}
Guy Cohen and Daoli Zhu.
\newblock Decomposition of nonsmooth optimization problems.
\newblock \emph{IFAC Proceedings Volumes}, 16\penalty0 (12):\penalty0 339--344,
  1983.

\bibitem[Cui and Shanbhag(2021)]{cui2021analysis}
Shisheng Cui and Uday~V Shanbhag.
\newblock On the analysis of variance-reduced and randomized projection
  variants of single projection schemes for monotone stochastic variational
  inequality problems.
\newblock \emph{Set-Valued and Variational Analysis}, 29:\penalty0 453--499,
  2021.

\bibitem[Cutkosky and Orabona(2019)]{cutkosky2019momentum}
Ashok Cutkosky and Francesco Orabona.
\newblock Momentum-based variance reduction in non-convex {SGD}.
\newblock \emph{Advances in neural information processing systems}, 32, 2019.

\bibitem[Diakonikolas(2020)]{diakonikolas2020halpern}
Jelena Diakonikolas.
\newblock Halpern iteration for near-optimal and parameter-free monotone
  inclusion and strong solutions to variational inequalities.
\newblock In \emph{Conference on Learning Theory}, pages 1428--1451. PMLR,
  2020.

\bibitem[Dieuleveut et~al.(2020)Dieuleveut, Durmus, and
  Bach]{dieuleveut2020bridging}
Aymeric Dieuleveut, Alain Durmus, and Francis Bach.
\newblock Bridging the gap between constant step size stochastic gradient
  descent and markov chains.
\newblock \emph{Annals of Statistics}, 2020.

\bibitem[Domke et~al.(2023)Domke, Gower, and Garrigos]{domke2023provable}
Justin Domke, Robert Gower, and Guillaume Garrigos.
\newblock Provable convergence guarantees for black-box variational inference.
\newblock \emph{Advances in neural information processing systems},
  36:\penalty0 66289--66327, 2023.

\bibitem[Duchi et~al.(2010)Duchi, Shalev-Shwartz, Singer, and
  Tewari]{duchi2010composite}
John~C Duchi, Shai Shalev-Shwartz, Yoram Singer, and Ambuj Tewari.
\newblock Composite objective mirror descent.
\newblock In \emph{COLT}, volume~10, pages 14--26. Citeseer, 2010.

\bibitem[Facchinei and Pang(2003)]{facchinei2003finite}
Francisco Facchinei and Jong-Shi Pang.
\newblock \emph{Finite-dimensional variational inequalities and complementarity
  problems}.
\newblock Springer, 2003.

\bibitem[Garrigos and Gower(2023)]{garrigos2023handbook}
Guillaume Garrigos and Robert~M Gower.
\newblock Handbook of convergence theorems for (stochastic) gradient methods.
\newblock \emph{arXiv preprint arXiv:2301.11235}, 2023.

\bibitem[Gladyshev(1965)]{gladyshev1965stochastic}
EG~Gladyshev.
\newblock On stochastic approximation.
\newblock \emph{Theory of Probability \& Its Applications}, 10\penalty0
  (2):\penalty0 275--278, 1965.

\bibitem[Gorbunov et~al.(2020)Gorbunov, Hanzely, and
  Richt{\'a}rik]{gorbunov2020unified}
Eduard Gorbunov, Filip Hanzely, and Peter Richt{\'a}rik.
\newblock A unified theory of sgd: Variance reduction, sampling, quantization
  and coordinate descent.
\newblock In \emph{International Conference on Artificial Intelligence and
  Statistics}, pages 680--690. PMLR, 2020.

\bibitem[Gorbunov et~al.(2022)Gorbunov, Berard, Gidel, and
  Loizou]{gorbunov2022stochastic}
Eduard Gorbunov, Hugo Berard, Gauthier Gidel, and Nicolas Loizou.
\newblock Stochastic extragradient: General analysis and improved rates.
\newblock In \emph{International Conference on Artificial Intelligence and
  Statistics}, pages 7865--7901. PMLR, 2022.

\bibitem[Gorbunov et~al.(2023)Gorbunov, Sadiev, Danilova, Horv{\'a}th, Gidel,
  Dvurechensky, Gasnikov, and Richt{\'a}rik]{gorbunov2023high}
Eduard Gorbunov, Abdurakhmon Sadiev, Marina Danilova, Samuel Horv{\'a}th,
  Gauthier Gidel, Pavel Dvurechensky, Alexander Gasnikov, and Peter
  Richt{\'a}rik.
\newblock High-probability convergence for composite and distributed stochastic
  minimization and variational inequalities with heavy-tailed noise.
\newblock \emph{arXiv preprint arXiv:2310.01860}, 2023.

\bibitem[Gower et~al.(2019)Gower, Loizou, Qian, Sailanbayev, Shulgin, and
  Richt{\'a}rik]{gower2019sgd}
Robert~Mansel Gower, Nicolas Loizou, Xun Qian, Alibek Sailanbayev, Egor
  Shulgin, and Peter Richt{\'a}rik.
\newblock Sgd: General analysis and improved rates.
\newblock In \emph{International conference on machine learning}, pages
  5200--5209. PMLR, 2019.

\bibitem[Grimmer(2019)]{grimmer2019convergence}
Benjamin Grimmer.
\newblock Convergence rates for deterministic and stochastic subgradient
  methods without lipschitz continuity.
\newblock \emph{SIAM Journal on Optimization}, 29\penalty0 (2):\penalty0
  1350--1365, 2019.

\bibitem[Gurbuzbalaban et~al.(2021)Gurbuzbalaban, Simsekli, and
  Zhu]{gurbuzbalaban2021heavy}
Mert Gurbuzbalaban, Umut Simsekli, and Lingjiong Zhu.
\newblock The heavy-tail phenomenon in sgd.
\newblock In \emph{International Conference on Machine Learning}, pages
  3964--3975. PMLR, 2021.

\bibitem[Halpern(1967)]{halpern1967fixed}
Benjamin Halpern.
\newblock Fixed points of nonexpanding maps.
\newblock \emph{Bulletin of the American Mathematical Society}, 73\penalty0
  (6):\penalty0 957--961, 1967.

\bibitem[Ilandarideva et~al.(2023)Ilandarideva, Juditsky, Lan, and
  Li]{ilandarideva2023accelerated}
Sasila Ilandarideva, Anatoli Juditsky, Guanghui Lan, and Tianjiao Li.
\newblock Accelerated stochastic approximation with state-dependent noise.
\newblock \emph{arXiv preprint arXiv:2307.01497}, 2023.

\bibitem[Jacobsen and Cutkosky(2023)]{jacobsen2023unconstrained}
Andrew Jacobsen and Ashok Cutkosky.
\newblock Unconstrained online learning with unbounded losses.
\newblock In \emph{International Conference on Machine Learning}, pages
  14590--14630. PMLR, 2023.

\bibitem[Karandikar and Vidyasagar(2023)]{karandikar2023convergence}
Rajeeva~L Karandikar and M~Vidyasagar.
\newblock Convergence rates for stochastic approximation: Biased noise with
  unbounded variance, and applications.
\newblock \emph{arXiv preprint arXiv:2312.02828}, 2023.

\bibitem[Khaled and Richt{\'a}rik(2023)]{khaled2022better}
Ahmed Khaled and Peter Richt{\'a}rik.
\newblock Better theory for {SGD} in the nonconvex world.
\newblock \emph{Transactions on Machine Learning Research}, 2023.

\bibitem[Khaled et~al.(2023)Khaled, Sebbouh, Loizou, Gower, and
  Richt{\'a}rik]{khaled2023unified}
Ahmed Khaled, Othmane Sebbouh, Nicolas Loizou, Robert~M Gower, and Peter
  Richt{\'a}rik.
\newblock Unified analysis of stochastic gradient methods for composite convex
  and smooth optimization.
\newblock \emph{Journal of Optimization Theory and Applications}, 199\penalty0
  (2):\penalty0 499--540, 2023.

\bibitem[Korpelevich(1976)]{korpelevich1976extragradient}
Galina~M Korpelevich.
\newblock The extragradient method for finding saddle points and other
  problems.
\newblock \emph{Matecon}, 12:\penalty0 747--756, 1976.

\bibitem[Lan and Zhou(2020)]{lan2020algorithms}
Guanghui Lan and Zhiqiang Zhou.
\newblock Algorithms for stochastic optimization with function or expectation
  constraints.
\newblock \emph{Computational Optimization and Applications}, 76\penalty0
  (2):\penalty0 461--498, 2020.

\bibitem[Lee and Kim(2021)]{lee2021fast}
Sucheol Lee and Donghwan Kim.
\newblock Fast extra gradient methods for smooth structured
  nonconvex-nonconcave minimax problems.
\newblock \emph{Advances in Neural Information Processing Systems},
  34:\penalty0 22588--22600, 2021.

\bibitem[Lin et~al.(2016)Lin, Rosasco, and Zhou]{lin2016iterative}
Junhong Lin, Lorenzo Rosasco, and Ding-Xuan Zhou.
\newblock Iterative regularization for learning with convex loss functions.
\newblock \emph{Journal of Machine Learning Research}, 17\penalty0
  (77):\penalty0 1--38, 2016.

\bibitem[Needell et~al.(2014)Needell, Ward, and Srebro]{needell2014stochastic}
Deanna Needell, Rachel Ward, and Nati Srebro.
\newblock Stochastic gradient descent, weighted sampling, and the randomized
  kaczmarz algorithm.
\newblock \emph{Advances in neural information processing systems}, 27, 2014.

\bibitem[Nemirovski and Yudin(1983)]{nemirovski1983problem}
Arkadi Nemirovski and David Yudin.
\newblock \emph{Problem complexity and method efficiency in optimization}.
\newblock Wiley- Interscience Series in Discrete Mathematics, John Wiley, 1983.

\bibitem[Nemirovski et~al.(2009)Nemirovski, Juditsky, Lan, and
  Shapiro]{nemirovski2009robust}
Arkadi Nemirovski, Anatoli Juditsky, Guanghui Lan, and Alexander Shapiro.
\newblock Robust stochastic approximation approach to stochastic programming.
\newblock \emph{SIAM Journal on optimization}, 19\penalty0 (4):\penalty0
  1574--1609, 2009.

\bibitem[Nesterov(2007)]{nesterov2007dual}
Yurii Nesterov.
\newblock Dual extrapolation and its applications to solving variational
  inequalities and related problems.
\newblock \emph{Mathematical Programming}, 109\penalty0 (2):\penalty0 319--344,
  2007.

\bibitem[Nesterov(2018)]{nesterov2018lectures}
Yurii Nesterov.
\newblock \emph{Lectures on convex optimization}, volume 137.
\newblock Springer, 2018.

\bibitem[Neu and Okolo(2024)]{neu2024dealing}
Gergely Neu and Nneka Okolo.
\newblock Dealing with unbounded gradients in stochastic saddle-point
  optimization.
\newblock In \emph{Proceedings of the 41st International Conference on Machine
  Learning}, pages 37508--37530, 2024.

\bibitem[Nguyen et~al.(2023)Nguyen, Ene, and Nguyen]{nguyen2023improved}
Ta~Duy Nguyen, Alina Ene, and Huy~L Nguyen.
\newblock Improved convergence in high probability of clipped gradient methods
  with heavy tails.
\newblock \emph{arXiv preprint arXiv:2304.01119}, 2023.

\bibitem[Orabona(2020)]{orabona2020last}
Francesco Orabona.
\newblock Last iterate of sgd converges (even in unbounded domains), 2020.
\newblock URL
  \url{https://parameterfree.com/2020/08/07/last-iterate-of-sgd-converges-even-in-unbounded-domains/}.

\bibitem[Park and Ryu(2022)]{park2022exact}
Jisun Park and Ernest~K Ryu.
\newblock Exact optimal accelerated complexity for fixed-point iterations.
\newblock In \emph{International Conference on Machine Learning}, pages
  17420--17457. PMLR, 2022.

\bibitem[Pethick et~al.(2023)Pethick, Fercoq, Latafat, Patrinos, and
  Cevher]{pethick2023solving}
Thomas Pethick, Olivier Fercoq, Puya Latafat, Panagiotis Patrinos, and Volkan
  Cevher.
\newblock Solving stochastic weak minty variational inequalities without
  increasing batch size.
\newblock In \emph{International Conference on Learning Representations}, 2023.

\bibitem[Polyak and Tsypkin(1973)]{polyak1973pseudogradient}
BT~Polyak and Ya~Z Tsypkin.
\newblock Pseudogradient adaptation and training algorithms.
\newblock \emph{Automation and remote control}, 34:\penalty0 45--67, 1973.

\bibitem[Robbins and Monro(1951)]{robbins1951stochastic}
Herbert Robbins and Sutton Monro.
\newblock A stochastic approximation method.
\newblock \emph{The annals of mathematical statistics}, pages 400--407, 1951.

\bibitem[Robbins and Siegmund(1971)]{robbins1971convergence}
Herbert Robbins and David Siegmund.
\newblock A convergence theorem for non negative almost supermartingales and
  some applications.
\newblock \emph{Optimizing methods in statistics}, pages 233--257, 1971.

\bibitem[Shamir and Zhang(2013)]{shamir2013stochastic}
Ohad Shamir and Tong Zhang.
\newblock Stochastic gradient descent for non-smooth optimization: Convergence
  results and optimal averaging schemes.
\newblock In \emph{International conference on machine learning}, pages 71--79.
  PMLR, 2013.

\bibitem[Telgarsky(2022)]{telgarsky2022stochastic}
Matus Telgarsky.
\newblock Stochastic linear optimization never overfits with
  quadratically-bounded losses on general data.
\newblock In \emph{Conference on Learning Theory}, pages 5453--5488. PMLR,
  2022.

\bibitem[Tran-Dinh(2024)]{tran2024halpern}
Quoc Tran-Dinh.
\newblock From halpern’s fixed-point iterations to nesterov’s accelerated
  interpretations for root-finding problems.
\newblock \emph{Computational Optimization and Applications}, 87\penalty0
  (1):\penalty0 181--218, 2024.

\bibitem[Tseng(2000)]{tseng2000modified}
Paul Tseng.
\newblock A modified forward-backward splitting method for maximal monotone
  mappings.
\newblock \emph{SIAM Journal on Control and Optimization}, 38\penalty0
  (2):\penalty0 431--446, 2000.

\bibitem[Vlatakis-Gkaragkounis et~al.(2024)Vlatakis-Gkaragkounis, Giannou,
  Chen, and Xie]{vlatakis2024stochastic}
Emmanouil~Vasileios Vlatakis-Gkaragkounis, Angeliki Giannou, Yudong Chen, and
  Qiaomin Xie.
\newblock Stochastic methods in variational inequalities: Ergodicity, bias and
  refinements.
\newblock In \emph{International Conference on Artificial Intelligence and
  Statistics}, pages 4123--4131. PMLR, 2024.

\bibitem[Wang and Bertsekas(2016)]{wang2016stochastic}
Mengdi Wang and Dimitri~P Bertsekas.
\newblock Stochastic first-order methods with random constraint projection.
\newblock \emph{SIAM Journal on Optimization}, 26\penalty0 (1):\penalty0
  681--717, 2016.

\bibitem[Wright and Recht(2022)]{wright2022optimization}
Stephen~J Wright and Benjamin Recht.
\newblock \emph{Optimization for data analysis}.
\newblock Cambridge University Press, 2022.

\bibitem[Xu(2004)]{xu2004viscosity}
Hong-Kun Xu.
\newblock Viscosity approximation methods for nonexpansive mappings.
\newblock \emph{Journal of Mathematical Analysis and Applications},
  298\penalty0 (1):\penalty0 279--291, 2004.

\bibitem[Yan and Xu(2022)]{yan2022adaptive}
Yonggui Yan and Yangyang Xu.
\newblock Adaptive primal-dual stochastic gradient method for
  expectation-constrained convex stochastic programs.
\newblock \emph{Mathematical Programming Computation}, 14\penalty0
  (2):\penalty0 319--363, 2022.

\bibitem[Yoon and Ryu(2021)]{yoon2021accelerated}
TaeHo Yoon and Ernest~K Ryu.
\newblock Accelerated algorithms for smooth convex-concave minimax problems
  with ${O} (1/k^2)$ rate on squared gradient norm.
\newblock In \emph{International Conference on Machine Learning}, pages
  12098--12109. PMLR, 2021.

\bibitem[Yoon and Ryu(2022)]{yoon2022accelerated}
TaeHo Yoon and Ernest~K Ryu.
\newblock Accelerated minimax algorithms flock together.
\newblock \emph{arXiv preprint arXiv:2205.11093}, 2022.

\bibitem[Zhang(2004)]{zhang2004solving}
Tong Zhang.
\newblock Solving large scale linear prediction problems using stochastic
  gradient descent algorithms.
\newblock In \emph{Proceedings of the twenty-first international conference on
  Machine learning}, page 116, 2004.

\bibitem[Zhao et~al.(2022)Zhao, Chen, Zhu, and Li]{zhao2022randomized}
Lei Zhao, Ding Chen, Daoli Zhu, and Xiao Li.
\newblock Randomized coordinate subgradient method for nonsmooth optimization.
\newblock \emph{arXiv preprint arXiv:2206.14981}, 2022.

\end{thebibliography}

\end{document}